\documentclass[a4paper,10pt,leqno]{amsart}
        \title{Geodesic flow for $\CAT(0)$-groups}
       \author{Arthur Bartels}
       \author{Wolfgang L\"uck}
      \address{Westf\"alische Wilhelms-Universit\"at M\"unster\\
               Mathematisches Institut\\
               Einsteinstr.~62,
               D-48149 M\"unster, Germany}
        \email{bartelsa@math.uni-muenster.de}
      \urladdr{http://www.math.uni-muenster.de/u/bartelsa}
        \email{lueck@math.uni-muenster.de}
      \urladdr{http://www.math.uni-muenster.de/u/lueck}
         \date{March, 2010}

\usepackage{pdfsync}
\usepackage{calc}
\usepackage{enumerate,amssymb}
\usepackage{color}
\usepackage[arrow,curve,matrix,tips,2cell]{xy}
  \SelectTips{eu}{10} \UseTips
  \UseAllTwocells
\usepackage{tikz}
\usepackage{pdfcolmk}

\DeclareMathAlphabet{\matheurm}{U}{eur}{m}{n}

\DeclareMathOperator{\id}{id}
\DeclareMathOperator{\im}{im}

\DeclareMathOperator{\Isom}{Isom}
\DeclareMathOperator{\Min}{Min}

\newcommand{\Fin}{{\mathcal{F}\text{in}}}

\newcommand{\VCyc}{{\mathcal{V}\mathcal{C}\text{yc}}}

  \newcommand{\IN}{\mathbb{N}}

  \newcommand{\IR}{\mathbb{R}}

  \newcommand{\IZ}{\mathbb{Z}}

  \newcommand{\calf}{\mathcal{F}}

  \newcommand{\calu}{\mathcal{U}}
  \newcommand{\calv}{\mathcal{V}}
  
  \newcommand{\calw}{\mathcal{W}}

\newcommand{\peri}[2]{\operatorname{per}^{#1}_{#2}}

\theoremstyle{plain}
\newtheorem{theorem}{Theorem}[section]

\newtheorem{lemma}[theorem]{Lemma}

\newtheorem{corollary}[theorem]{Corollary}
\newtheorem{proposition}[theorem]{Proposition}

\newtheorem{convention}[theorem]{Convention}
\newtheorem*{theorem*}{Theorem}
\newtheorem*{mtheorem*}{Main Theorem}

\theoremstyle{definition}
\newtheorem{definition}[theorem]{Definition}

\newtheorem{remark}[theorem]{Remark}
\newtheorem{notation}[theorem]{Notation}

\theoremstyle{remark}
\newtheorem*{summary*}{Summary}

\makeatletter\let\c@equation=\c@theorem\makeatother

\newenvironment{numberlist}
  {\begin{list}{}%
   {%
    \setlength{\leftmargin}{\labelwidth+\labelsep}%
   }%
  }%
  {\end{list}}

\newcommand{\x}{{\times}}

\newcommand{\dd}{{\partial}}
\newcommand{\e}{{\varepsilon}}
\newcommand{\CAT}{{\operatorname{CAT}}}
\newcommand{\FS}{\mathit{FS}}

\newcommand{\hyp}{\mathit{hyp}}

\begin{document}

\begin{abstract}
We associate to a $\CAT(0)$-space  a flow space  that 
can be used as the replacement for the geodesic flow on 
the sphere tangent bundle of a Riemannian manifold.
We use this flow space to
prove that $\CAT(0)$-group are transfer reducible over
the family of virtually cyclic groups.
This result is an important ingredient in our proof
of the Farrell-Jones Conjecture for these groups.
\end{abstract}

\maketitle

\newlength{\origlabelwidth}
\setlength\origlabelwidth\labelwidth


\typeout{---------- Introduction --------}

\section*{Introduction}

In~\cite{Bartels-Lueck(2009borelhyp)} we introduced the concept of 
\emph{transfer reducible groups} with respect to a family of subgroups.
This definition is somewhat technical and recalled as 
Definition~\ref{def:transfer-reducible} below.
We showed that groups that are transfer reducible over
the family of virtually cyclic subgroups satisfy
the Farrell-Jones Conjecture
with coefficients in an additive category.
For further explanations about the Farrell-Jones Conjecture we refer
for instance to~\cite{Bartels-Lueck(2009borelhyp),
Bartels-Lueck-Reich(2008appl), Lueck-Reich(2005)}, where more information
about the applications, history, literature and status is given.

By a $\CAT(0)$-group we mean a group $G$ that admits a cocompact 
proper action by isometries on a finite dimensional $\CAT(0)$-space.
The following is our main result in this paper and has already been 
used in~\cite{Bartels-Lueck(2009borelhyp)}.

\begin{mtheorem*} \label{the:main_theorem}
  Every $\CAT(0)$-group is transfer reducible over the
  family of virtually cyclic subgroups.
\end{mtheorem*}

A similar result for hyperbolic groups has been proven
in~\cite[Proposition~2.1]{Bartels-Lueck(2009borelhyp)} using the technical
paper~\cite{Bartels-Lueck-Reich(2008cover)}, where an important input
is the flow space for hyperbolic groups due to Mineyev~\cite{Mineyev(2005)}.
The methods for hyperbolic groups cannot be 
transferred directly to $\CAT(0)$-groups,
but the general program is the same and carried out in this paper. 

An important step in the proof of the  theorem above is the
construction of a flow space $\FS(X)$ associated to $\CAT(0)$-spaces $X$,
which is a replacement for the geodesic flow
on the sphere tangent bundle of a Riemannian manifold of
non-positive curvature.
In particular, the dynamic of the flow on $\FS(X)$ is 
similar to the geodesic flow.
As in the hyperbolic case, the flow space $\FS(X)$ is not a bundle
over $X$. 

In Section~\ref{sec:A-flow-space-metric-space} 
we assign to any metric space $X$
its flow space $\FS(X)$ (see Definition~\ref{def_flow_space_FS(X)}). 
Elements in $\FS(X)$ are generalized geodesics, i.e.,
continuous maps $c \colon \IR \to X$ such that either 
$c$ is constant or there exists 
$c_-, c_+ \in \overline{\IR} = \IR \amalg \{-\infty, \infty\}$
with $c_- < c_+$ such that $c$ is locally constant outside the 
interval $(c_-,c_+)$ and its restriction to $(c_-,c_+)$ is an 
isometric embedding. 
The flow on $\FS(X)$ is given by
$\Phi_{\tau}(c)(t) := c(t+ \tau)$.
The topology on $\FS(X)$ is the topology of uniform convergence on compact
subsets; this is also the topology associated to a natural metric
on $\FS(X)$. 
Many properties of $X$ can be transported to $\FS(X)$.
For example, if a group $G$ acts by isometries on $X$, then there is
an induced isometric action on $\FS(X)$.
If the action on $X$ is cocompact, 
then the induced action is also cocompact.

In Sections~\ref{sec:The_flow_space_associated_to_a_CAT(0)-space},%
~\ref{sec:bounded-period}
and~\ref{sec:Contracting-transfers-FS}
we study properties of $\FS(X)$ under the assumption that $X$ is a
$\CAT(0)$-space.
The main observation in 
Section~\ref{sec:The_flow_space_associated_to_a_CAT(0)-space}
is that the endpoint evaluation maps $c \mapsto c(\pm \infty)$ 
from  $\FS(X)$ to the bordification of $\overline{X}$ of $X$
are continuous on the complement of the subspace $\FS(X)^\IR$
of constant generalized geodesics.
These are used in Proposition~\ref{prop:embedding-of-FS} 
to give coordinates on $\FS(X) - \FS(X)^\IR$
and allow a detailed study of the topology of $\FS(X) - \FS(X)^\IR$.
We also discuss in Subsection~\ref{subsec:example_non-pos-curv}
the case where $X$ is a non-positively curved manifold.
In Section~\ref{sec:Contracting-transfers-FS}
we prove our main flow estimates for $\FS(X)$.
These are crucial for our main result and differ
from the corresponding estimates in the hyperbolic case.
In the hyperbolic case the flow acts contracting on
geodesics that determine the same point at infinity.
This is not true in the $\CAT(0)$-situation, for instance on flats
the flow acts as an isometry.
This problem is overcome by using a variant of the focal transfer 
(formulated as a homotopy action) from~\cite{Farrell-Jones(1993c)}. 
In Section~\ref{sec:bounded-period} we assume
that $G$ acts properly and by isometries on $X$ and study
the periodic orbits of $\FS(X)$ with respect to the induced action.
In Theorem~\ref{the:covering_of_FS(X)_gamma} we construct
certain open covers for the subspace 
$\FS(X)_{\leq \gamma}$ of $\FS(X)$ consisting 
of all geodesics of period $\leq \gamma$.
The dimension of this cover is uniformly bounded and the cover
is long in the sense that for every $c \in \FS(X)_{\leq \gamma}$
there is a member $U$ of this cover that contains $\Phi_{[-\gamma,\gamma]}(c)$. 
(In fact, $U$ will contain even $\Phi_\IR(c)$.)
This result is much harder than the corresponding result in the
hyperbolic case, because in the $\CAT(0)$-case the periodic
orbits are no longer discrete, but appear in continuous families.
  
Section~\ref{sec:flow-spaces-and-S-long-covers} contains the final
preparation for the proof of our main theorem.
We show in Proposition~\ref{prop:wide-covers-with-flow-spaces}
that the existence of a suitable flow space for a group $G$ 
implies that $G$ is transfer reducible over a given family.
This result depends very much on the long thin covers for 
flow spaces from~\cite{Bartels-Lueck-Reich(2008cover)}. 
 
In Section~\ref{sec:Non-positively_curved_groups_are_transfer_reducible}
we put our previous results together and prove our main theorem.
It is only here that we assume that the action of
$G$ on the $\CAT(0)$-space $X$ is cocompact.
All previous results are formulated without this assumption.
This forces for example the appearance of a compact set $K$
in Theorem~\ref{the:covering_of_FS(X)_gamma}
and in Definition~\ref{def:at-infinity_plus_periods}.
There are of course prominent groups, for instance $\mathit{SL}_n(\IZ)$,
that are naturally equipped with an isometric proper action on 
a $\CAT(0)$-space, where the action is not cocompact. 
We hope that the level of generality in 
Sections~\ref{sec:A-flow-space-metric-space} 
to~\ref{sec:flow-spaces-and-S-long-covers} will be useful
to prove the Farrell-Jones Conjecture for some of these 
groups.


\subsection*{Conventions}
Let $H$ be a (discrete) group that acts on a space $Z$.  
We will say that the
action is \emph{proper}, if for any $x \in X$ there is an open 
neighborhood $U$
such that $\{h \in H \mid h \cdot U \cap U \neq \emptyset \}$ is finite.  
If $Z$
is locally compact, this is equivalent to the 
condition that for any compact
subset $K \subset Z$ the set 
$\{h \in H \mid h \cdot K \cap K \neq \emptyset\}$
is finite.  

We will say that the action is \emph{cocompact} if the quotient
space $H\backslash Z$ is compact. 
If $Z$ is locally compact, this is equivalent
to the existence of a compact subset $L \subseteq Z$ 
such that $G \cdot L = Z$.

Let $X$ be a topological space. Let $\calu$ be an open covering.
Its \emph{dimension} $\dim(\calu) \in \{0,1,2, \ldots \} \amalg \{\infty\}$ 
is the infimum over all integers $d \ge 0$
such that for any collection $U_0$, $U_1$, \ldots , $U_{d+1}$ of pairwise
distinct elements in $\calu$ the intersection 
$\bigcap_{i=0}^{d+1} U_i$ is empty.
An open covering $\calv$ is a refinement of $\calu$ 
if for every $V \in \calv$ there
is $U \in \calu$ with $V \subseteq U$. 
The \emph{(topological) dimension} (sometimes also called
\emph{covering dimension}) of a topological space $X$
$$\dim(X) \in \{0,1,2, \ldots \} \amalg \{\infty\}$$ 
is the infimum over all integers $d \ge 0$ such that any open covering
$\calu$ possesses a refinement $\calv$ with $\dim(\calv) \le d$.

For  a metric space $Z$, 
a subset $A \subseteq Z$ and $\epsilon > 0$, we set
\begin{eqnarray*}
    B_{\epsilon}(A)
    & := &
    \{z \in Z \mid \exists z' \in A \;\text{with } 
                              d_{Z}(z,z')  < \epsilon\}.
    \\
    \overline{B}_{\epsilon}(A)
    & := &
    \{z \in Z \mid \exists z' \in A \;\text{with } 
                              d_{Z}(z,z')  \leq \epsilon\}.
\end{eqnarray*}
We abbreviate $B_\e(z) = B_\e(\{z\})$ and
$\overline{B}_\e(z) = \overline{B}_\e(\{z\})$.
A metric space $Z$ is called \emph{proper} 
if for every $R > 0$ and every point $z \in Z$ the closed ball 
$\overline{B}_R(z)$ of radius $R$ around $z$  is compact.
A map is called \emph{proper} if the
preimage of any compact subset is again compact.

A \emph{family  of subgroups of a group $G$} is a set of subgroups closed
under conjugation and taking subgroups. Denote by $\VCyc$ the family of
virtually cyclic subgroups.


\subsection*{Homotopy actions and $S$-long covers}
Next we explain the notion of transfer reducible. 
In~\cite{Bartels-Lueck(2009borelhyp)} we introduced 
the following definitions in
order to formulate conditions on groups that imply that a group satisfies the
Farrell-Jones conjecture in $K$- and $L$-theory.

\begin{definition}[Homotopy $S$-action]
\label{def:S-action_plus_long-covers}
Let $S$ be a finite subset of a group $G$.
Assume that $S$ contains the trivial element $e \in G$.
Let $X$ be a space.
\begin{enumerate}
\item \label{def:S-action_plus_long-covers:action}
      A \emph{homotopy $S$-action $(\varphi,H)$ on $X$} consists of
      continuous maps $\varphi_g \colon X \to X$
      for $g \in S$ and homotopies
      $H_{g,h} \colon X \times [0,1] \to X$
      for $g,h \in S$ with $gh \in S$ such that
      $H_{g,h}(-,0) = \varphi_g \circ \varphi_h$
      and $H_{g,h}(-,1) = \varphi_{gh}$ holds for $g,h \in S$ with $gh \in S$.
      Moreover, we require that $H_{e,e}(-,t) = \varphi_e = \id_X$
      for all $t \in [0,1]$;
\item \label{def:S-action_plus_long-covers:F}
      Let $(\varphi,H)$ be a homotopy $S$-action on $X$.
      For $g \in S$ let $F_g(\varphi,H)$ be the set of all
      maps $X \to X$ of the form $x \mapsto H_{r,s}(x,t)$
      where $t \in [0,1]$ and $r,s \in S$ with $rs = g$;
\item \label{def:S-action_plus_long-covers:S(g,x)}
      Let $(\varphi,H)$ be a homotopy $S$-action on $X$.
      For $(g,x) \in G \times X$ and $n \in \IN$,
      let $S^n_{\varphi,H}(g,x)$ be the subset
      of $G \times X$ consisting of all $(h,y)$
      with the following property:
      There are
      $x_0,\dots,x_n \in X$,
      $a_1,b_1,\dots,a_n,b_n \in S$,
      $f_1,\tilde{f}_1,\dots,f_n,\tilde{f_n} \colon X \to X$,
      such that
      $x_0 = x$, $x_n = y$, $f_i \in F_{a_i}(\varphi,H)$,
      $\tilde{f}_i \in F_{b_i}(\varphi,H)$,
      $f_i(x_{i-1}) = \tilde{f}_i(x_i)$ and
      $h = g a_1^{-1} b_1 \dots a_n^{-1} b_n$;
\item \label{def:S-action_plus_long-covers:long}
      Let $(\varphi,H)$ be a homotopy $S$-action on $X$
      and $\calu$ be an open cover of $G \times X$.
      We say that $\calu$ is \emph{$S$-long with respect to $(\varphi,H)$}
      if for every $(g,x) \in G \times X$ there is $U \in \calu$ containing
      $S^{|S|}_{\varphi,H}(g,x)$ where $|S|$ is the cardinality of $S$.
\end{enumerate}
\end{definition}

\begin{definition}[$N$-dominated space]
  \label{def:N-dominated_space}
  Let $X$ be a metric space and $N \in \IN$.
  We say that $X$ is \emph{controlled $N$-dominated}
  if for every $\e > 0$ there is a finite $CW$-complex $K$ of dimension
  at most $N$, maps $i \colon X \to K$, $p \colon K \to X$
  and a homotopy $H \colon X \times [0,1] \to X$ between $p \circ i$ and
  $\id_X$ such that for every $x \in X$
  the diameter of $\{ H(x,t) \mid t \in [0,1] \}$  is
  at most $\e$.
\end{definition}

\begin{definition}[Open $\calf$-cover]
  \label{def:F-cover}
  Let $Y$ be a $G$-space. Let $\calf$ be a family of  subgroups of $G$.
  A subset $U \subseteq Y$ is called
  an \emph{$\calf$-subset} if 
  \begin{enumerate}
  \item For $g  \in G$ and $U \in  \calu$ we have
      $g(U) = U$ or $U \cap g(U) = \emptyset$, where
      $g(U):=  \{gx \mid x \in U \}$;
  \item The subgroup
      $G_U := \{ g \in G \; | \; g(U) = U \}$ lies in $\calf$.
  \end{enumerate}
  An \emph{open $\calf$-cover} of $Y$ is a collection $\calu$ of open
  $\calf$-subsets of $Y$  such that the following conditions are satisfied:
  \begin{enumerate}
  \item $Y = \bigcup_{U \in \calu} U$;
  \item For $g \in  G$, $U \in \calu$  the set
      $g(U) $ belongs to $\calu$.
  \end{enumerate}
\end{definition}

\begin{definition}[Transfer reducible] \label{def:transfer-reducible}
  Let $G$ be a group and $\calf$ be a family of subgroups.
  We will say that $G$ is \emph{transfer reducible} over $\calf$ if
  there is a number $N$ with the following property:

  For every finite subset $S$ of $G$ there are
  \begin{itemize}
  \item a contractible compact controlled $N$-dominated
      metric space $X$;
  \item a homotopy $S$-action $(\varphi,H)$ on $X$;
  \item a cover $\calu$ of $G \times X$ by open sets,
  \end{itemize}
  such that the following holds for the $G$-action on $G \times X$ given by
  $g\cdot (h,x) = (gh,x)$:
  \begin{enumerate}
  \item $\dim \calu \leq N$;
  \item $\calu$ is $S$-long with respect to $(\varphi,H)$;
  \item $\calu$ is an open $\calf$-covering.
  \end{enumerate}
\end{definition}


\subsection*{Acknowledgements} 
The authors thank Holger Reich for fruitful discussions
on $\mathit{SL}_n(\IZ)$.
These led us to formulate some of the results of this
paper without an cocompactness assumption.
The first author thanks Tom Farrell for explaining to him the structure
of closed geodesics in non-positively curved manifolds 
as used in~\cite{Farrell-Jones(1991)} in connection with
the Farrell-Jones Conjecture.
The second author wishes to thank the Max-Planck-Institute
and the Hausdorff Research Institute for Mathematics at Bonn for their hospitality
during longer stays in 2007 and 2009 when parts of this paper were written.
This paper is financially supported by the Leibniz-Preis of the second author.


\typeout{------ A flow space associated to a metric space ------}

\section{A flow space associated to a metric space}
\label{sec:A-flow-space-metric-space}

\begin{summary*}
  In this section we introduce the flow space $\FS(X)$ for
  arbitrary metric spaces.
  We show that $\FS(X)$ is a proper metric
  space if $X$ is a proper metric space
  (see Proposition~\ref{prop:FS-is-proper}). 
  If $X$ comes with a proper cocompact
  isometric $G$-action, then $\FS(X)$ inherits a proper cocompact isometric
  $G$-action  (see Proposition~\ref{prop:cocompact}). 
\end{summary*}

\begin{definition} \label{def:generalized_geodesic}
  Let $X$ be a metric space.
  A continuous map $c \colon \IR \to X$ is called a 
  \emph{generalized  geodesic} 
  if there are  
  $c_-,c_+ \in \overline{\IR} := \IR \coprod \{-\infty,\infty\}$ 
  satisfying
  \[
   c_-  \le  c_+, \quad  c_- \not= \infty, \quad  c_+  \not=  -\infty,  
  \]
  such that $c$ is locally constant on the complement of the interval 
  $I_c := (c_{-},c_{+})$ and restricts to an isometry on $I_c$.
\end{definition}

The numbers $c_-$ and $c_+$ are uniquely determined by $c$, provided
that $c$ is not constant. 

\begin{definition} \label{def_flow_space_FS(X)}
  Let $(X,d_X)$ be a metric space.
  Let $\FS = \FS(X)$ be the set of all generalized geodesics in $X$.
  We define a metric on $\FS(X)$ by
  $$
   d_{\FS(X)}(c,d) := 
     \int_\IR \frac{d_{X}(c(t),d(t))}{2 e^{|t|}} \; dt.
  $$
  Define a flow 
  $$
  \Phi \colon \FS(X) \times \IR \to \FS(X)
  $$
  by $\Phi_{\tau}(c)(t) = c(t + \tau)$ for 
  $\tau \in \IR$, $c \in \FS(X)$ and $t \in \IR$.
\end{definition}

The integral 
$\int_{-\infty}^{+\infty} \frac{d_{X}(c(t),d(t))}{2 e^{|t|}} \; dt$ 
exists since $d_{X}(c(t),d(t)) \le 2|t| + d_{X}(c(0),d(0))$ 
by the triangle inequality.
Obviously $\Phi_{\tau}(c)$ is a generalized geodesic with
\begin{eqnarray*}
\Phi_{\tau}(c)_- & = & c_- - \tau;
\\
\Phi_{\tau}(c)_+ & = & c_+ - \tau,
\end{eqnarray*}
where $- \infty - \tau  := - \infty$ and $\infty - \tau := \infty$.

We note that any isometry $(X,d_X) \to (Y,d_Y)$ induces
an isometry $\FS(X) \to \FS(Y)$ by composition. 
In particular, the isometry group of $(X,d_X)$ acts canonically
on $\FS(X)$. 
Moreover, this action commutes with the flow.

For a general metric space $X$ all generalized geodesics
may be constant.
Later we will consider the case where $X$ is a $\CAT(0)$-space,
but in the remainder of this section we will consider properties
of $\FS(X)$ that do not depend on the $\CAT(0)$-condition.

\begin{lemma} \label{lem:Phi_well-defined}
  Let $(X,d_X)$ be a metric space.
  The map $\Phi$ is a continuous flow and we have
  for $c,d \in \FS(X)$ and $\tau,\sigma \in \IR$
  \begin{eqnarray*}
    d_{\FS(X)}\bigl(\Phi_{\tau}(c), \Phi_{\sigma}(d)\bigr)
    & \le  &  
    e^{|\tau|} \cdot d_{\FS(X)}(c,d) + |\sigma - \tau|.
  \end{eqnarray*}
\end{lemma}

\begin{proof}
  Obviously $\Phi_{\tau}\circ \Phi_{\sigma} = \Phi_{\tau + \sigma}$ for 
  $\tau, \sigma \in \IR$ and $\Phi_0= \id_{\FS(X)}$.
  The main task is to show that
  $\Phi \colon \FS(X) \times \IR \to \FS(X)$ is continuous.

  We estimate for $c \in \FS(X)$ and $\tau \in \IR$
  \begin{eqnarray*}
    d_{\FS(X)}\bigl(c, \Phi_{\tau}(c)\bigr)  
    & = & 
    \int_{\IR} \frac{d_X\bigl(c(t),c(t+\tau)\bigr)}{2e^{|t|}} \; dt
    \\
    & \le &
    \int_{\IR} \frac{|\tau|}{2e^{|t|}} \; dt 
    \\ 
    & = &
    |\tau | \cdot \int_{\IR} \frac{1}{2e^{|t|}} \; dt
    \\ 
    & = &
    |\tau|.
  \end{eqnarray*}
  We estimate for $c,d \in \FS(X)$ and $\tau \in \IR$
  \begin{eqnarray*}
    d_{\FS(X)}\bigl(\Phi_{\tau}(c), \Phi_{\tau}(d)\bigr)
    & = & 
    \int_{\IR}\frac{d_X\big(c(t + \tau),d(t+\tau)\bigr)}{2e^{|t|}} \; dt
    \\
    & = & 
    \int_{\IR}\frac{d_X\big(c(t),d(t)\bigr)}{2e^{|t-\tau|}} \; dt
    \\
    & \le  & 
    \int_{\IR}\frac{d_X\big(c(t),d(t)\bigr)}{2e^{|t|-|\tau|}}  \; dt
    \\
    & = & e^{|\tau|} \cdot 
      \int_{\IR}\frac{d_X\big(c(t),d(t)\bigr)}{2e^{|t|}} \; dt
    \\
    & = & e^{|\tau|} \cdot d_{\FS(X)}(c,d).
  \end{eqnarray*}
  The two inequalities above together with the triangle inequality
  imply for $c,d \in \FS(X)$ and $\tau, \sigma \in \IR$
  \begin{eqnarray*}
    \lefteqn{d_{\FS(X)}\bigl(\Phi_{\tau}(c), \Phi_{\sigma}(d)\bigr)}
    &&
    \\
    & = &  
    d_{\FS(X)}\bigl(\Phi_{\tau}(c), \Phi_{\sigma-\tau}\circ \Phi_{\tau}(d)\bigr)
    \\
    & \le  &  
    d_{\FS(X)}\bigl(\Phi_{\tau}(c), \Phi_{\tau}(d)\bigr) 
      + d_{\FS(X)}\bigl(\Phi_{\tau}(d), 
               \Phi_{\sigma-\tau} \circ \Phi_{\tau}(d)\bigr)
    \\
    & \le &
    e^{|\tau|} \cdot d_{\FS(X)}(c,d) + |\sigma - \tau|.
  \end{eqnarray*}
  This implies that $\Phi$ is continuous  at $(c,\tau)$.
\end{proof}

The following lemma relates distance in $X$ to distance in $\FS(X)$.

\begin{lemma}\label{lem:comparing_d_FS-d_X}
  Let $c,d \colon \IR \to X$ be generalized geodesics. 
  Consider $t_0 \in \IR$.
  \begin{enumerate}
  \item \label{lem:comparing_d_FS-d_X:general}
        $d_X\bigl(c(t_0),d(t_0)\bigr)  \; \le \;
             e^{|t_0|} \cdot d_{\FS}(c,d) + 2$;
  \item \label{lem:comparing_d_FS-d_X:small}
        If $d_\FS(c,d) \leq 2 e^{-|t_0|-1}$, then
        \[
        d_X\bigl(c(t_0),d(t_0)\bigr) \; \le \;
              \sqrt{4e^{|t_0|+1}} \cdot \sqrt{d_\FS(c,d)}.
        \]
        In particular, $c \mapsto c(t_0)$ defines
        a uniform continuous map $\FS(X) \to X$.
  \end{enumerate}
\end{lemma}

\begin{proof}
  We abbreviate $D := d_X(c(t_0),d(t_0))$.
  Since $c$ and $d$ are generalized geodesics, we conclude using the 
  triangle inequality
  for $t \in \IR$
  \begin{equation*}
    d_X(c(t),d(t)) 
     \ge 
    D - d_X\bigl(c(t_0),c(t)\bigr) - 
       d_X\bigl(d(t_0),d(t)\bigr) 
     \ge 
    D - 2\cdot |t-t_0|.
  \end{equation*}
  This implies
  \begin{eqnarray}
    d_{\FS(X)}(c,d) 
    & = & 
    \int_{-\infty}^{+\infty} \frac{d_{X}(c(t),d(t))}{2 e^{|t|}} \; dt
    \nonumber
    \\
    & \ge & 
    \int_{-D/2+t_0}^{D/2+t_0} 
    \frac{D - 2\cdot |t-t_0|}{2 e^{|t|}} \; dt
    \nonumber 
    \\
    & = & 
    \int_{-D/2}^{D/2} 
    \frac{D - 2\cdot |t|}{2 e^{|t+t_0|}} \; dt
    \nonumber
    \\
    & \ge & 
    \int_{-D/2}^{D/2} 
    \frac{D - 2\cdot |t|}{2 e^{|t|+|t_0|}} \; dt
    \nonumber
    \\
    & = & 
    e^{-|t_0|} \cdot \int_{-D/2}^{D/2} 
    \frac{D - 2\cdot |t|}{2 e^{|t|}} \; dt
    \nonumber
    \\
    & = &  
    e^{-|t_0|} \cdot \int_0^{D/2}  
      \bigl(D - 2t\bigr) 
            \cdot e^{-t} \; dt
    \nonumber
    \\
    & = &  
    e^{-|t_0|} \cdot  
    \left[(-D+2+2t)  \cdot e^{-t} \right]_0^{D/2} 
    \nonumber
    \\
    & = & 
    e^{-|t_0|} \cdot   \left(2 \cdot e^{-D/2} +  D  - 2\right).
    \label{seq_inequalities}
    \end{eqnarray}
    Since 
    $ e^{-|t_0|} \cdot   \left(2 \cdot e^{-D/2} +  D  - 2\right) \geq 
    e^{-|t_0|} \cdot ( D - 2)$, assertion~\ref{lem:comparing_d_FS-d_X:general} 
    follows. It remains to prove assertion~\ref{lem:comparing_d_FS-d_X:small}.

  Consider the function $f(x) = e^{-x} + x - 1 - \frac{x^2}{2e}$. 
  We have $f'(x) = - e^{-x} + 1 - \frac{x}{e}$ and $f''(x) = e^{-x} - e^{-1}$.
  Hence $f''(x) \ge 0$ for $x \in [0,1]$. 
  Since $f'(0) = 0$, this implies $f'(x) \ge 0$ for $x \in [0,1]$. 
  Since $f(0) = 0$, this implies $f(x) \ge 0$ for $x \in [0,1]$. 
  Setting $x = D/2$ we obtain $\frac{(D/2)^2}{2e} \leq  e^{-D/2} + D/2 -1$
  for $D \in [0,2]$. From~\eqref{seq_inequalities} we get
  $d_\FS(c,d) \geq 2e^{-|t_0|} \cdot \left( e^{-D/2} + D/2 -1 \right)$.
  Therefore
  $$
  \frac{D^2}{4e^{|t_0| + 1}} \le d_{\FS(X)}(c,d) 
  \quad \;\text{if } D \le 2.
  $$
  Consider the function $g(x) = e^{-x} + x - 1$. 
  Since $g'(x) =  - e^{-x} + 1 > 0$ for $x \ge 1$,
  we conclude $g(x) > g(1) = e^{-1}$ for $x > 1$. 
  Hence~\eqref{seq_inequalities} implies
  $$
  d_{\FS(X)}(c,d) >  2e^{-|t_0| - 1} \quad \;\text{if } 
       D > 2.
  $$
  Hence we have
  $$
  d_X\bigl(c(t_0),d(t_0)\bigr) = D \; \le \; \sqrt{4e^{|t_0| +1}} \cdot 
                                                   \sqrt{d_{\FS}(c,d)}.
  $$
  if $d_{\FS}(c,d) \le 2\cdot e^{-|t_0|-1}$.
\end{proof}

\begin{lemma}
  \label{lem:c-plus-minus-cont}
  Let $(X,d_X)$ be a metric space.
  The maps
  \begin{eqnarray*}
    FS(X) - \FS(X)^{\IR} & \to & \overline{\IR}, \quad c \mapsto c_-;
    \\
    FS(X) - \FS(X)^{\IR} & \to & \overline{\IR}, \quad c \mapsto c_+,
  \end{eqnarray*}
  are continuous.
\end{lemma}

\begin{proof}
  By an obvious symmetry it suffices to consider the second map. 
  Let $c \in FS(X) - \FS(X)^{\IR}$.
  Let $\alpha_0 \in \IR$ with $\alpha_0 < c_+$.
  We will first show that there is $\e_0$ such that
  $d_+ > \alpha_0$ if $d_\FS(c,d) < \e_0$.
  Pick $s_0,t_0 \in \IR$ such that $\alpha_0 < s_0 < t_0 < c_+$ and 
  $c(s_0) \neq c(t_0)$.
  By Lemma~\ref{lem:comparing_d_FS-d_X}~\ref{lem:comparing_d_FS-d_X:small} 
  there is $\e_0 > 0$
  such that $d_\FS(c,d) < \e_0$ implies
  \[
  \max \left\{ d_X\bigl(c(s_0),d(s_0)\bigr), d_X\bigl(c(t_0),d(t_0)\bigr) \right\} 
                          < \frac{d_X\bigl(c(s_0),c(t_0)\bigr)}{3}.
  \]
  For $d$ with $d_\FS(c,d) < \e_0$ we have 
  $d_X \bigl( d(s_0), d(t_0) \bigr) > d_X \bigl( c(s_0), c(t_0) \big) /3$ 
  by the triangular inequality
  and in particular, $d(s_0) \neq d(t_0)$.
  This implies $d_+ > s_0$ and therefore $d_+ > \alpha_0$.
  If $c_+ = + \infty$, then this shows that the second map is
  continuous at $c$.

  If $c_+ < \infty$, then we need to show in addition that for a given 
  $\alpha_1 > c_+$, there is $\e_1 > 0$ such that $d_+ < \alpha_1$
  for all $d$ with $d_\FS(c,d) < \e_1$.
  Note that the previous argument also implied that
  $d_- < t_0 < c_+$ if $d_\FS(c,d) < \e_0$ 
  (because then $d(s_0) \neq d(t_0)$).
  Pick now $s_1, t_1 \in \IR$ with $c_+ < s_1 < t_1 < \alpha_1$. 
  By Lemma~\ref{lem:comparing_d_FS-d_X}~\ref{lem:comparing_d_FS-d_X:small} 
  there is $\e_1 \in \IR$ satisfying  $0 < \e_1 < \e_0$
  such that $d_\FS(c,d) < \e_1$ implies
  \[
  \max \left\{ d_X\bigl(c(s_1),d(s_1)\bigr), d_X\bigl(c(t_1),d(t_1)\bigr) \right\} 
                          < \frac{t_1 - s_1}{2}.
  \]  
  Because $c(s_1) = c(t_1)$ we get
  $d_X\bigl(d(s_1),d(t_1)\bigr) < t_1 - s_1$ for $d$ with $d_\FS(c,d) < \e_1$.
  This implies $d_+ < t_1$ or $d_- > s_1$, 
  because otherwise $d_X\bigl(d(s_1),d(t_1)\bigr) = t_1 - s_1$.
  However, $d_- < c_+ < s_1$ because $\e_1 < \e_0$.
  Thus $d_+ < t_1 < \alpha_1$. 
\end{proof}

\begin{proposition}
  \label{prop:uniform-convergence-on-compact}
  Let $(X,d_X)$ be a metric space. Let 
  $(c_n)_{n \in \IN}$ be a sequence in $\FS(X)$.
  Then it converges uniformly on compact subsets to $c \in \FS(X)$
  if and only if it converges to $c$ with respect to $d_{\FS(X)}$.
\end{proposition}

\begin{proof}
  From Lemma~\ref{lem:comparing_d_FS-d_X}~\ref{lem:comparing_d_FS-d_X:small}
  we conclude that convergence with respect to $d_{\FS(X)}$ implies
  uniform convergence on compact subsets.
  
  Let $(c_n)_{n \in \IN}$ be a sequence in $\FS(X)$ 
  that converges uniformly on compact subsets to $c \in \FS(X)$.
  Let $\e > 0$.
  Pick $\alpha \geq 1$ such that 
  $\int_\alpha^{\infty} \frac{2t+\e}{e^t}dt < \e$.
  Because of uniform convergence on $[-\alpha,\alpha]$, 
  there is $n_0$ such that $d_X(c_n(t),c(t)) \leq \e / \alpha$ for all
  $n \geq n_0$, $t \in [-\alpha,\alpha]$.
  In particular, $d_X(c_n(t),c(t)) \leq \e + 2|t|$ for all $t$,
  provided $n \geq n_0$.
  Thus for $n \geq n_0$
  \begin{eqnarray*}
    d_\FS(c_n, c) 
    & = & \int_{-\infty}^{\infty} \frac{d_X\bigl(c_n(t),c(t)\bigr)}{2e^{|t|}}dt
    \\
    & \leq & \int_{-\alpha}^\alpha \frac{\e / \alpha}{2e^{|t|}}dt
             + 2 \int_{\alpha}^\infty \frac{\e + 2t}{2e^{|t|}}dt
    \\
    & \leq &  \e +  \e = 2 \e. 
  \end{eqnarray*}
  This shows $c_n \to c$ with respect to $d_\FS$, because
  $\e$ was arbitrary.
\end{proof}

\begin{lemma}
  \label{lem:FS-closed}
  Let $(X,d_X)$ be a metric space.
  The flow space $\FS(X)$ is sequentially closed in the space of all maps
  $\IR \to X$ with respect to the topology of 
  uniform convergence on compact subsets. 
\end{lemma}

\begin{proof}
  Let $(c_n)_{n \in \IN}$ be a sequence of generalized geodesics
  that converges uniformly on compact subsets to $f \colon \IR \to X$.
  We have to show that $f$ is a generalized geodesic.
  By passing to a subsequence we can assume that either 
  $c_n \in \FS^\IR$ for all $n$, or $c_n \not\in \FS^\IR$ for all $n$.
  In the first case $f \in \FS^\IR$.
  Thus it remains to treat the second case.
  In this case we have well-defined sequences $(c_n)_{-}$ and
  $(c_n)_{+}$.
  After passing to a further subsequence we can assume that
  these sequences converge in $[-\infty,\infty]$.
  Thus there are $\alpha_{-},\alpha_{+} \in [-\infty,\infty]$
  such that $(c_n)_{\pm} \to \alpha_{\pm}$ as $n \to \infty$.
  We will show that $f$ is a generalized geodesic with
  $f_{\pm} = \alpha_{\pm}$ or $f \in \FS^\IR$.
  Clearly, $d_X(f(s), f(t)) \leq |t-s|$ for all $s,t$. 
  
  If $\alpha_{-} > -\infty$, then we have to show
  that $f(s) = f(t)$ for all $s < t \leq \alpha_{-}$.
  Pick $\e > 0$.
  There is $n_0$ such that $|f(\tau) - c_n(\tau)| < \e$
  for all $\tau \in [s-\e,t]$, $n \geq n_0$.
  Since $(c_n)_{-} \to \alpha_{-}$, there is $k \geq n_0$ such 
  that $s - \e,t - \e \leq (c_k)_{-}$.
  Thus
  \begin{eqnarray*}
    d_X(f(s),f(t)) & \leq &  d_X(f(s-\e),f(t-\e)) + 2\e 
    \\
    & \leq & d_X(c_k(s-\e),c_k(t-\e)) + 4\e
    \\
    & = &  4\e.     
  \end{eqnarray*}
  Because $\e$ is arbitrary, we conclude $f(s) = f(t)$. 
  If $\alpha_+ < \infty$, then a similar argument shows 
  that $f(s) = f(t)$ for all $s,t \geq \alpha_{+}$.
  If $\alpha_{-} = \alpha_{+}$, then $f \in \FS(X)^{\IR}$ and we are done.
  It remains to treat the case $\alpha_{-} < \alpha_{+}$. We have to show
  that $d_X(f(s),f(t)) = t-s$ for all $s,t \in \IR$ with
  $\alpha_{-} \leq s < t \leq \alpha_{+}$. 
  Pick $\e > 0$, such that $2\epsilon  < t - s$.
  There is $n_0$ such that $|f(\tau) - s_n(\tau)| < \e$
  for all $\tau \in [s,t]$, $n \geq n_0$.
  Since $(c_n)_{\pm} \to \alpha_{\pm}$, there is $k \geq n_0$
  such that $(c_k)_{-} \leq s + \e < t - \e \leq (c_k)_{+}$.
  Thus
  \begin{eqnarray*}
    d_X(f(s),f(t)) & \geq & d_X(f(s+\e),f(t-\e)) - 2\e 
    \\
    & \geq & d_X(c_k(s+\e),c_k(t-\e)) - 4\e
    \\
    & = & ( (t - \e) - (s + \e) ) - 4\e
    \\
    & = & t - s - 6\e.         
  \end{eqnarray*}
  This implies $d_X(f(s),f(t)) = t-s$, because  $\e$ was arbitrarily 
  small.  
\end{proof}

\begin{proposition}
  \label{prop:FS-is-proper}
  If $(X,d_X)$ is a proper metric space,
  then $(\FS(X),d_{FS(X)})$ is a proper metric space.
\end{proposition}

\begin{proof}
  Let $R > 0$ and $c \in \FS(X)$. It suffices to show
  that the closed ball $\overline{B}_R(c)$
  in $\FS(X)$ is sequentially compact since any metric space satisfies
  the second countability axiom.
  Let $(c_n)_n \in \IN$ be a sequence in $\overline{B}_R(c)$.
  By Lemma~\ref{lem:comparing_d_FS-d_X}~\ref{lem:comparing_d_FS-d_X:general}
  there is $R' > 0$ such that $c_n(0) \in \overline{B}_{R'}(c(0))$.
  By assumption $\overline{B}_{R'}(c(0))$ is compact.
  Now we can apply the Arzel\`a-Ascoli theorem 
  (see for example~\cite[I.3.10, p.36]{Bridson-Haefliger(1999)}).
  Thus after passing to a subsequence there is $d \colon \IR \to X$
  such that $c_n \to d$ uniformly on compact subsets.
  Lemma~\ref{lem:FS-closed} implies $d \in \FS(X)$.
\end{proof}

\begin{lemma}
  \label{lem:evaluation-is-proper}   Let $(X,d_X)$ be a 
   proper metric space and $t_0 \in \IR$. 
   Then the evaluation map $\FS(X) \to X$ defined by
   $c \mapsto c(t_0)$ is uniformly continuous and proper. 
\end{lemma}

\begin{proof}
  The map is uniformly continuous by 
  Lemma~\ref{lem:comparing_d_FS-d_X}~\ref{lem:comparing_d_FS-d_X:small}.
  To show that is is also proper, it suffices by 
  Proposition~\ref{prop:FS-is-proper} to show that preimages 
  of closed balls have finite diameter.
  If $d_X(c(t_0),d(t_0)) \le r$, then 
  $d_X(c(t),d(t)) \leq r + 2 |t - t_0|$.
  Thus
  \begin{equation*}
    d_\FS(c,d) \leq \int_\IR \frac{r + 2 |t-t_0|}{2e^{|t|}}dt
     \quad \mbox{provided} \quad d_X\bigl(c(t_0),d(t_0)\bigr) \leq r.
  \end{equation*}
\end{proof}

\begin{proposition}
  \label{prop:cocompact}
  Let $G$ act isometrically and proper on the proper metric 
  space $(X,d_X)$.
  Then the action of $G$ on $(\FS(X),d_\FS)$ is also
  isometric and proper.
  If the action of $G$ on $X$ is in addition cocompact,
  then this is also true for the action on $\FS(X)$.
\end{proposition}

\begin{proof}
  The action of $G$ on $\FS(X)$ is isometric.
  The map $\FS(X) \to X$ defined by $c \mapsto c(0)$ 
  is $G$-equivariant, continuous and proper by 
  Lemma~\ref{lem:evaluation-is-proper}.
  The existence of such a map implies that the $G$-action
  on $\FS(X)$ is also proper.
  This also implies that the action of $G$ on $\FS(X)$
  is cocompact, provided that the action on $X$ is cocompact.
\end{proof}

\begin{lemma}
  \label{lem:FS-IR-is-closed}
  Let $(X,d_X)$ be a metric space.
  Then $\FS(X)^\IR$ is closed in $\FS(X)$.
\end{lemma}

\begin{proof}
  Note that $\FS(X)^\IR$ is the space of constant generalized geodesics.
  Let $c \in \FS(X) - \FS(X)^\IR$.
  Pick $t_0$, $t_1 \in \IR$ such that $c(t_0) \neq c(t_1)$.
  Set $\delta := d_X(c(t_0),c(t_1)) / 2$.
  For $x \in X$ then  $d_X(x,c(t_0)) \geq \delta$
  or $d_X(x,c(t_1)) \geq \delta$.
  Denote by $c_x$ the constant generalized geodesic at $x$.
  If $d_X(x,c(t_0)) \geq \delta$, then
  $d_X(x,c(t)) \geq \delta/2$ if $t \in [t_0-\delta/2,t_0+\delta/2]$.
  Thus in this case
  \begin{equation*}
    d_\FS(c_x,c) \geq \int_{t_0 - \delta/2}^{t_0 + \delta/2} 
           \frac{\delta/2}{2e^{|t|}}dt =: \e_0.
  \end{equation*}
  Similarly, 
  \begin{equation*}
    d_\FS(c_x,c) \geq \int_{t_1 - \delta/2}^{t_1 + \delta/2} 
           \frac{\delta/2}{2e^{|t|}}dt =: \e_1,
  \end{equation*}
  if $d_X(x,c(t_1)) \geq \delta/2$.
  Hence $B_{\e}(c) \cap \FS(X)^\IR = \emptyset$
  if we set $\e := \min \{ \e_0 /2,\e_1/2 \}$.
\end{proof}

\begin{notation}
  \label{not:FS-0}
  Let $X$ be a metric space.
  For $c \in \FS(X)$ and $T \in [0,\infty]$, define $c|_{[-T,T]} \in \FS(X)$ by
  \begin{equation*}
    c|_{[-T,T]}(t) := \begin{cases}
                      c(-T) & \quad \mbox{if} \;  t \le -T; \\
                       c(t) & \quad \mbox{if} \; -T \le t \le T; \\ 
                       c(T) & \quad \mbox{if} \; t \ge T.
                     \end{cases}       
  \end{equation*}
  Obviously $c|_{[-\infty,\infty]} = c$ and
  if $c \notin \FS(X)^\IR$ and $(-T,T) \cap (c_-,c_+) \neq \emptyset$ 
  then
  $\bigl(c|_{[-T,T]}\bigr)_{-} = \max\{c_{-}, -T \}$ and 
  $\bigl(c|_{[-T,T]}\bigr)_{+} = \min\{c_{+}, T\}$.

  We denote by 
  $$\FS(X)_f := \left\{ \left.c \in \FS(X)-\FS(X)^\IR \; \right|\; 
                     c_{-} > -\infty,c_+ < \infty \right\} \cup
                                               \FS(X)^\IR$$
  the subspace of finite geodesics.
\end{notation}

\begin{lemma}
  \label{lem:push-to-finite-geodesic}
  Let $(X,d_X)$ be a metric space.
  The map $H \colon \FS(X) \x [0,1] \to \FS(X)$
  defined by $H_\tau(c) := c|_{[\ln(\tau),-\ln(\tau)]}$
  is continuous and satisfies $H_0 = \id_{\FS(X)}$ and
  $H_\tau(c) \in \FS(X)_f$ for $\tau > 0$.
\end{lemma}

\begin{proof}
  Observe that for $c \in \FS(X)$ and $T,T' \geq 0$
  \begin{equation*}
    d_X\bigl(c_{[-T,T]}(t),c_{[-T',T']}(t)\bigr) \leq |T-T'| 
        \quad \mbox{for all} \; t \in \IR. 
  \end{equation*}
  Recall from Proposition~\ref{prop:uniform-convergence-on-compact}
  that the topology on $\FS(X)$ is the topology
  of uniform convergence on compact subsets.
  Let $c_n \to c$ uniformly on compact subsets, and 
  $\tau_n \to \tau$.
  Let $\alpha > 0$.
  We need to show that 
  $c_n|_{[\ln(\tau_n),-\ln(\tau)]} \to c_{[\ln(\tau),-\ln(\tau)]}$
  uniformly on $[-\alpha,\alpha]$.
  
  Consider first the case $\tau = 0$.
  Then $c = c|_{[\ln(\tau),-\ln(\tau)]}$.
  Moreover, $-\ln(\tau_n) > \alpha$ for sufficient large $n$. 
  Thus $c_n(t) = c_n|_{[\ln(\tau_n),-\ln(\tau)]}(t)$ for 
  such $n$ and $t \in [-\alpha,\alpha]$. 
  This implies the claim for $\tau = 0$.

  Next consider the case $\tau \in (0,1]$.
  Let $\e > 0$.
  There is $n_0 = n_0(\e ) \in \IN$ such that 
  \begin{equation*}
    \begin{array}{lclcl}
       d_X(c_n(t),c(t)) & \leq & \e 
       & \mbox{if} & n \geq n_0, t \in [-\alpha,\alpha]; 
       \\[1ex]
       |\ln(\tau) - \ln(\tau_n)| & \leq & \e
       & \mbox{if} & n \geq n_0.  
    \end{array}
  \end{equation*}
  Then for $t \in [-\alpha,\alpha]$, $n \geq n_0$,
  \begin{eqnarray*}
    \lefteqn{d_X\bigl(c_n|_{[\ln(\tau_n),-\ln(\tau_n)]}(t),
                     c|_{[\ln(\tau),-\ln(\tau)]}(t)\bigr)} & & \\
    & \leq &
    d_X\bigl(c_n|_{[\ln(\tau_n),-\ln(\tau_n)]}(t),c_n|_{[\ln(\tau),-\ln(\tau)]}(t)\bigr)
    \\ & & \hspace{5em}
    + d_X\bigl(c_n|_{[\ln(\tau),-\ln(\tau)]}(t),c|_{[\ln(\tau),-\ln(\tau)]}(t)\bigr)
    \\
    & \leq &
    2 \e.
  \end{eqnarray*}
  This implies the claim in the second case because $\e$ was arbitrary.
\end{proof}


\typeout{------ The flow space associated to a CAT(0)-space ------}

\section{The flow space associated to a CAT(0)- space}
\label{sec:The_flow_space_associated_to_a_CAT(0)-space}

\begin{summary*}
  In this section we study $\FS(X)$ further in the 
  case where $X$ is a $\CAT(0)$-space. 
  Let $\overline{X}$ be the bordification of $X$.  
  We construct an injective continuous map from
  $\FS(X) - \FS(X)^\IR$ to 
  $\overline{\IR} \times \overline{X} \times X \times \overline{X}
    \times \overline{\IR}$ which is a homeomorphism onto its image
  (see~Proposition~\ref{prop:embedding-of-FS}). 
  It sends a generalized geodesic $c$ to 
  $\bigl(c_-,c(-\infty),c(0),c(\infty),c_+\bigr)$, where
  $c(-\infty)$ and $c(\infty)$ are the two endpoints of $c$. 
  This is used to show that $\FS(X) - \FS(X)^\IR$ 
  has finite dimension if $X$ has
  (see Proposition~\ref{pro:dim-for_FS-FSR}), and that 
  $\FS(X)-\FS(X)^{\IR}$ is locally connected 
  (see Proposition~\ref{prop:flow-space-is-locally-connected}).
  We will relate our construction to the geodesic flow on the sphere
  tangent bundle of a simply connected Riemannian manifold with
  non-positive sectional curvature in
  Subsection~\ref{subsec:example_non-pos-curv}.
\end{summary*}

For the definition of a $\CAT(0)$-space we refer
to~\cite[II.1.1 on p.~158]{Bridson-Haefliger(1999)}, 
namely to be a geodesic space all
of whose geodesic triangles satisfy the $\CAT(0)$-inequality.  
We will follow the notation and the description of  
the \emph{bordification} $\overline{X} = X \cup \partial X$ of
a $\CAT(0)$-space $X$ given in~\cite[Chapter II.8]{Bridson-Haefliger(1999)}.  
The definition of the topology of this bordification is briefly reviewed in 
Remark~\ref{rem:cone-topology}.
In this section we will use the
following convention.

\begin{convention}
  \label{conv:cat-0-satisfy}
Let 
\begin{itemize}
 \item $X$ be a complete $\CAT(0)$-space;
 \item $\overline{X} := X \cup \dd X$ be the bordification of $X$,
      see \cite[Chapter~II.8]{Bridson-Haefliger(1999)}.  
\end{itemize}
\end{convention}


\subsection{Evaluation of generalized geodesics at infinity}
\label{subsec:evaluation-at-infinity}

\begin{definition} \label{def:c-infty}
  For $c \in \FS(X)$  we set $c(\pm \infty) := \lim_{t \to \pm \infty} c(t)$,
  where the limit is taken in $\overline{X}$.
\end{definition}

Since $X$ is by assumption a $\CAT(0)$-space,
we can find for  $x_- \in X$ and $x_+ \in \overline{X}$  
a generalized geodesic
$c \colon \IR \to X$ with $c(\pm \infty) = x_{\pm}$ 
(see~\cite[II.8.2 on p.~261]{Bridson-Haefliger(1999)}).
It is not true in general that for two different points $x_-$ and 
$x_+$ in $\partial X$ there is a geodesic
$c$ with $c(-\infty) = x_-$ and $c(\infty) = x_+$.

\begin{remark}
  [Cone topology on $\overline{X}$.]
  \label{rem:cone-topology}
  A \emph{generalized geodesic ray} is a generalized geodesic
  $c$ that is either a constant generalized geodesic or a non-constant
  generalized geodesic with $c_- = 0$. 
  Fix a base point $x_0 \in X$.
  For every $x \in \overline{X}$, there is a unique
  generalized geodesic ray $c_x$ such that $c(0) = x_0$
  and $c(\infty) = x$,  
  see~\cite[II.8.2 on p.261]{Bridson-Haefliger(1999)}.
  Define for $r > 0$
  \begin{equation*}
    \rho_r = \rho_{r,x_0} \colon \overline{X} \to \overline{B}_r(x_0)      
  \end{equation*}
  by $\rho_r(x) := c_x(r)$.
  The sets $(\rho_r)^{-1}(V)$ with $r > 0$, 
  $V$ an open subset of $\overline{B}_r(x_0)$
  are a basis for the cone topology on $\overline{X}$,
  see~\cite[II.8.6 on p.263]{Bridson-Haefliger(1999)}.
  A map $f$ whose target is $\overline{X}$ is continuous if and
  only if $\rho_r \circ f$ is continuous for all $r$.
  The cone topology is independent of the choice of base point,
  see~\cite[II.8.8 on p.264]{Bridson-Haefliger(1999)}.
\end{remark}

\begin{lemma}
  \label{lem:c-plus-minus-infinty-cont}
  The maps
  \begin{eqnarray*}
    FS(X) - \FS(X)^{\IR} & \to & \overline{X}, \quad c \mapsto c(-\infty);
    \\
    FS(X) - \FS(X)^{\IR} & \to & \overline{X}, \quad c \mapsto c(\infty),
  \end{eqnarray*}
  are continuous.
\end{lemma}

The proof of this Lemma depends on the following result.

\begin{lemma}
  \label{lem:sigma_t-converge}
  Given $\e > 0$, $a > 0$ and $s > 0$, there exists a constant 
  $T = T(\e,a,s) > 0$ such that the following in true: 
  if $x$, $x' \in X$ with $d_X(x,x') \leq a$,
  if $c \colon \IR \to X$ is a 
  generalized geodesic ray with $c(0) = x$, and if
  $\sigma_t \colon [0,d(x',c(t)] \to X$ is the geodesic 
  from $x'$ to $c(t)$, then $d_X(\sigma_t(s),\sigma_{t+t'}(s)) < \e$
  for all $t \geq T$ and all $t' \geq 0$.
\end{lemma}

\begin{proof}
  In~\cite[II.8.3 on p.261]{Bridson-Haefliger(1999)} this is
  proven under the additional assumptions that $c$ is
  a geodesic ray and that $d_X(x,x') = a$.
  But the proof given in~\cite{Bridson-Haefliger(1999)}
  can be adapted as follows to give our more general result.

  The argument given in~\cite{Bridson-Haefliger(1999)}
  can be applied without change to show that there is 
  $T$ such that $d_X(\sigma_t(s),\sigma_{t+t'}(s)) < \e$
  for all $t \geq T$, $t' \geq 0$
  provided that $t+t' \leq c_+$.
  (This is needed to deduce that
  $|t-a| \leq d_X(x',c(t))$ and that
  $|t+t'-a| \leq d_X(x',c(t+t')$.)
  
  It remains to treat the case where $t \geq T$, 
  $t' \geq 0$ and $t + t' \geq c_{+}$.
  If $t \geq c_{+}$, then $\sigma_t = \sigma_{t+t'}$
  (because $c(t+t') = c(t) = c(c_+)$)
  and there is nothing to show.
  Thus we can assume $t \leq c_{+}$. 
  Set $t'' := c_{+} - t$.
  Then $t+t'' \leq c_{+}$,
  $t'' \geq 0$ and $\sigma_{t+t'} = \sigma_{t+t''}$ 
  (because $c(t + t') = c(t + t'') = c(c_{+})$).
  Thus $d_X(\sigma_t(s), \sigma_{t+t'}(s)) = 
          d_X(\sigma_t(s),\sigma_{t+t''}(s)) < \e$.      
\end{proof}

\begin{proof}
  [Proof of Lemma~\ref{lem:c-plus-minus-infinty-cont}]
  By an obvious symmetry 
  it suffices to consider the second map. 
  Let $c \in \FS(X) - \FS(X)^\IR$.
  Set $s_0 := \max\{0,c_{-}\} + 1$ and $x_0 := c(s_0)$.
  Then
  \begin{equation*}
    \widetilde{c} \colon \IR \to X, \quad t \mapsto \begin{cases}
  c(t + s_0) & t \ge 0;\\ c(s_0) & t \le 0,
  \end{cases}
  \end{equation*}
  is a generalized geodesic ray starting at $x_0$.
  We need to show that the map $f_r \colon \FS(X) \to \overline{B}_r(x_0)$
  defined by $f_r(d) := \rho_r(d(\infty))$ is for all $r > 0$
  continuous at $c$, see Remark~\ref{rem:cone-topology}.
  Note that $f_r(c) = \widetilde{c}(r) = c(s_0 + r)$.

  Let $\e > 0$ be given.
  By Lemma~\ref{lem:c-plus-minus-cont} there is $\delta_0$
  such that $d_{-} < s_0$ for all $d$ with $d_\FS(c,d) < \delta_0$.
  In particular, we obtain for any such generalized geodesic ray
  $d$ a generalized geodesic ray $\widetilde{d}$
  by putting $\widetilde{d}(t) = d(t+s_0)$ for $t \ge 0$ and 
  $\widetilde{d}(t) = d(s_0)$ for $t \le 0$.
  
  For $t > 0$ and $d \in \FS$ with $d_\FS(c,d) < \delta_0$ denote by 
  $\sigma^d_t \colon [0,d_X(x_0,d(s_0 + t))] \to X$ 
  the geodesic from $x_0$ to $d(s_0 + t)$.
  By Lemma~\ref{lem:sigma_t-converge} there is a number $T$ 
  (not depending on $d$!) such that
  $d_X(\sigma^d_t(r),\sigma^d_{t+t'}(r)) < \e$ for all $t' > 0$, $t \geq T$,
  provided that $d_X(d(s_0),x_0) \leq 1$.
  We extend $\sigma^d_t$ to a generalized geodesic ray by setting 
  $\sigma^d_t(s) := d(s_0 + t)$ for $s > d_X(x_0,d(s_0 + t))$ 
  and  $\sigma_t^d(s) := x_0$ for $s < 0$.
  The unique generalized geodesic ray $c_{d(\infty)}$ from $x_0$
  to $d(\infty)$ can be constructed as the  limit
  of the $\sigma^d_t$, 
  see~\cite[Proof of (8.2) on p.262]{Bridson-Haefliger(1999)}.
  It follows that $\sigma^d_t(r) \to c_{d(\infty)}(r)$ as $t \to \infty$.
  By definition of $\rho_r$ we have $c_{d(\infty)}(r) = f_r(d)$.
  Therefore $d_X(\sigma_T(r),f_r(d)) \leq \e$,
  provided that $d_X(d(s_0),c(s_0)) \leq 1$.

  By Lemma~\ref{lem:comparing_d_FS-d_X}~\ref{lem:comparing_d_FS-d_X:small} 
  there exists $0 < \delta_1 < \delta_0$ such that 
  \[
    d_X( c(s_0), d(s_0) ) < 1 \quad \mbox{and} \quad 
    d_X(c(s_0+T),d(s_0+T)) < \e
  \] 
  if $d_\FS(c,d) < \delta_1$.
  Consider the triangle whose
  vertices are $x_0 = c(s_0)$, $c(s_0+T)$ and $d(s_0+T)$.
  Recall that $\sigma_T^d$ is side of this triangle
  that connects $x_0$ to $d(s_0+T)$.
  Using the $\CAT(0)$-condition in this triangle
  it can be deduced that if $d_\FS(c,d) < \delta_1$
  then
  \begin{equation*}
     d_X( c(s_0 + r), \sigma^d_T(r) ) < 2\e.
  \end{equation*}
  
  Therefore for $d_\FS(c,d) < \delta_1$ we conclude
  \begin{eqnarray*}
    d_X( f_r(c), f_r(d) ) & = & d_X( c(r+s_0) , c_{d(\infty}(r) )
    \\
    & \leq & d_X( c(r + s_0), \sigma^d_T(r) ) 
       + d_X ( \sigma^d_T(r), c_{d(\infty)}(r) )
    \\
    & < & 3 \e. 
  \end{eqnarray*}
  Because $\e$ was arbitrary this implies that 
  $f_r$ is continuous at $c$.
\end{proof}


\subsection{Embeddings of the flow space}
\label{subsec:flow-space-proper}

\begin{proposition}
  \label{prop:embedding-of-FS}
  If $X$ is proper as a metric space, then
  the map 
  \[
  E \colon \FS(X) - \FS(X)^\IR \to \overline{R} \x \overline{X} \x X \x
                         \overline{X} \x \overline{R}
  \]
  defined by 
  $E(c) := (c_{-},c(-\infty),c(0),c(\infty),c_+)$
  is injective and continuous. It is a homeomorphism onto its image.
\end{proposition}

\begin{proof} 
  Lemmas~\ref{lem:comparing_d_FS-d_X}~\ref{lem:comparing_d_FS-d_X:small},~%
\ref{lem:c-plus-minus-cont} and~\ref{lem:c-plus-minus-infinty-cont}
  imply that $E$ is continuous.

  Next we show that $E$ is injective.
  Let $c \in \FS(X) - \FS(X)^\IR$.
  If $t_0 \in [c_{-},c_{+}]$, then $t \mapsto c(t+t_0)$, $t \geq 0$ is
  the unique~\cite[II.8.2 on p.261]{Bridson-Haefliger(1999)} 
  (generalized) geodesic ray from $c(t_0)$ to $c(\infty)$,
  and similarly $t \mapsto c(t_0-t)$, $t \geq 0$ is the unique
  (generalized) geodesic ray from $c(t_0)$ to $c(-\infty)$.
  Let $c$, $d \in \FS(X) - \FS(X)^\IR$, with 
  $c(\pm \infty) = d(\pm \infty)$, $c_\pm = d_\pm$.
  Then $c$ and $d$ will agree if and only if $c(t) = d(t)$
  for some $t \in [c_{-},c_{+}]$, $t \neq \pm \infty$.
  If $E(c) = E(d)$, then there is such a $t$: if one of $c_{-}$
  and $c_{+}$ is real (not $\pm \infty$), then it can be used,
  otherwise $t = 0$ works.

  It remains to show that the inverse $E^{-1}$ to $E$ defined on the 
  image of $E$ is continuous.
  Let $(c_n)_{n \in \IN}$ be a sequence in $\FS(X)  - \FS(X)^\IR$
  and $c \in \FS(X)  - \FS(X)^\IR$ such that $E(c_n) \to E(c)$
  as $n \to \infty$.
  We have to show that $c_n \to c$ as $n \to \infty$.
  We proceed by contradiction and assume this fails.
  Then there is a subsequence $c_{n_k}$ and $\rho > 0$ such
  that $d_\FS( c, c_{n_k}) > \rho$ for all $k$.
  We can pass to this subsequence and assume $d_\FS(c, c_n) > \rho$ 
  for all $n$.
  We have $c_n(0) \to c(0)$ as $n \to \infty$.
  The evaluation at $t = 0$ is a proper map
  $\FS(X) \to X$ by Lemma~\ref{lem:evaluation-is-proper}.
  Thus we can pass to a further subsequence and assume that
  $c_n \to d$ as $n \to \infty$ with $d \in \FS(X)$.
  
  We claim that $d \not\in \FS(X)^\IR$.
  Because $c \not\in \FS(X)^\IR$ we have either
  $c(-\infty) \neq c(0)$ or $c(\infty) \neq c(0)$.
  By symmetry we may assume $c(-\infty) \neq c(0)$.
  We consider now two different cases.\\
  First case: $c_{-} \neq -\infty$.
  Then $c_{-} \in \IR$ and we can consider the evaluation at 
  $c_{-}$.
  We have $(c_n)_{-} \to c_{-}$,
  $c_n( -\infty) \to c(-\infty)$ and $c_n(0) \to c(0)$
  since $E(c_n) \to E(c)$ as $n \to \infty$.
  Moreover $c_n(c_{-}) \to d(c_{-})$ and $c_n(0) \to d(0)$ 
  since $c_n \to d$ as $n \to \infty$.
  Therefore $c(0) = d(0)$ and we get
  \begin{eqnarray*}
    d_X\bigl( d(c_{-}), c(c_{-}) \bigr) 
    & \leq & d_X\bigl( d(c_{-}), c_n(c_{-})\bigr) 
               + d_X\bigl(c_n(c_{-}), c_n( (c_n)_{-} ) \bigr)
           \\ & & + d_X\bigl(c_n( (c_n)_{-}), c( c_{-})\bigr)
    \\
    & \leq & d_X\bigl( d(c_{-}), c_n(c_{-}) \bigr)
               + \bigl| c_{-} - (c_n)_{-} \bigr|
               + d_X\bigl( c_n( -\infty ), c(-\infty)\bigr)
    \\
    & \to & 0 \; \text{as $n \to \infty$}.
  \end{eqnarray*}
  Thus $d(c_{-}) = c(c_{-}) = c(-\infty) \neq c(0) = d(0)$.
  Therefore $d \not\in \FS(X)^\IR$.\\
  Second case: $c_{-} = -\infty$.
  Because $(c_n)_{+} \to c_{+} \neq -\infty$, there is $K > 0$ such that
  $-K < (c_n)_{+}$ for all $n$.
  Since $(c_n)_{-} \to c_{-} = -\infty$ 
  we have $(c_n)_{-} < -2K$ for all sufficiently large $n$. 
  Then 
  \begin{eqnarray*}
  d_X\bigl(c_n(-2K),c_n(0)\bigr) 
  & = & 
  d_X\bigl( c_n(-2K), c_n(-K) \bigr) + 
  d_X\bigl( c_n(-K), c_n(0)\bigr)
  \\
  & \geq &
  d_X\bigl( c_n(-2K), c_n(-K) \bigr) 
  \\
  & = & K
\end{eqnarray*}
  for sufficiently large $n$.
  Using Lemma~\ref{lem:evaluation-is-proper} we conclude 
  $d_X\bigl(d(-2K),d(0)\bigr) \geq K$.
  Therefore $d \not\in \FS(X)^\IR$.
  This finishes the proof of the claim.
  
  Because $d \in \FS(X) - \FS(X)^\IR$ we can apply $E$ to $d$ and deduce
  $E(d) = E(c)$ from continuity of $E$. 
  Thus $c = d$ because $E$ is injective.
  This contradicts  $d_\FS(c, c_n) > \rho$ 
  for all $n$ and finishes the proof.
\end{proof}

Recall that $\FS(X)_f$ is the subspace of 
finite geodesics, see Notation~\ref{not:FS-0}.

\begin{proposition}
  \label{prop:local-embeding-of-FS_f}
  Assume that $X$ is proper as a metric space.
  Then the map
  \[
  E_f \colon \FS(X)_f - \FS(X)^\IR \to \IR \x X \x X  
  \]
  defined by
  $E_f(c) = \bigl(c_{-},c(-\infty),c(\infty)\bigr)$
  is a homeomorphism onto its image
  \[
  \im E_f = \{ (r,x,y) \mid x \neq y \}.
  \]
  In particular,
  $\FS(X)_f - \FS(X)^\IR$ is locally path connected.
\end{proposition}

\begin{proof}
  Lemmas~\ref{lem:c-plus-minus-cont} and~\ref{lem:c-plus-minus-infinty-cont}
  imply that $E_f$ is continuous.  
  The map $E_f$ is injective with the stated image because of existence 
  and uniqueness of geodesics between points in $X$,
  see~\cite[II.1.4 on p.160]{Bridson-Haefliger(1999)},
 
  Next we show that the induced map
  $$E_f \colon \FS(X)_f - \FS(X)^{\IR} \to  \{(r,x,y) \mid x \neq y \}$$
  is proper. Let $K \subset \{(r,x,y) \mid x \neq y \}$ 
  be compact.  We will show that ${E_f}^{-1}(K)$ is
  sequentially compact.  Let $(c_n)_{n \in \IN}$ be a sequence in
  ${E_f}^{-1}(K)$.  After passing to a subsequence, we may assume that
  $E_f(c_n)$ converges in K.  Thus $(c_n)_{-} \to t_0 \in \IR$, $c_n(-\infty)
  \to x_{-} \in X$, $c_n(\infty) \to x_{+} \in X$, and $x_{-} \neq x_{+}$.  
  We have 
  \begin{eqnarray*}
  d_X(c_n(t_0),x_{-}) 
  & \leq &
  d_X(c_n(t_0),c_n(-\infty)) + d_X(c_n(-\infty),x_{-}) 
  \\
  & \leq &
  |t_0 - (c_n)_{-}| +   d_X(c_n(-\infty),x_{-}).
  \end{eqnarray*}
  Thus $c_n(t_0) \to x_{-}$ as $n \to \infty$.  Using
  Lemma~\ref{lem:evaluation-is-proper} we deduce that $(c_n)_{n \in \IN}$ 
  has a convergent subsequence in
  $\FS(X)$, that we will again denote by $c_n$.  
  So now $c_n \to c$ in $\FS(X)$
  for some $c \in \FS(X)$.

  We will show next that $c \not\in \FS^{\IR}$.  We have $(c_n)_{+} - (c_n)_{-}
  = d_X(c_n(\infty),c_n(-\infty)) \to d_X(x_{-},x_{+})$, as $n \to \infty$.
  Thus $(c_n)_{+} \to t_1 := t_0 + d_X(x_{-},x_{+})$.  We have
  $$
  d_X(c_n(t_1),x_{+}) \leq d_X(c_n(t_1),c_n(\infty)) + d_X(c_n(\infty),x_{+})
  \leq |t_1 - (c_n)_{+}| + d_X(c_n(\infty),x_{+}).
  $$
  Thus $c_n(t_1) \to x_{+}$
  as $n \to \infty$.  From Lemma~\ref{lem:evaluation-is-proper} we conclude
  $$
  c(t_1) = \lim c_n(t_1) = x_{+} \neq  x_{-} = c(t_0).
  $$
  Thus $c(t_1) \neq c(t_0)$ and $c \in \FS - \FS^{\IR}$.

  Now $E_f(c_n) \to E_f(c)$ as $E_f$ is continuous.  Therefore $c_{-} = t_0$,
  $c(-\infty) = x_{-}$, $c(\infty) = x_{+}$.  Thus $c \in {E_f}^{-1}(K)$.
  Hence $E_f \colon \FS(X)_f - \FS(X)^{\IR} \to  \{(r,x,y) \mid x \neq y \}$
  is an injective continuous  proper maps of metric spaces. This implies
  that it is a homeomorphism (see~\cite[2.2 and~2.7]{Steenrod(1967)}).

  By the existence of geodesics the image of $E_f$ is locally path
  connected. Hence $\FS(X)_f - \FS(X)^\IR$ is locally path connected.
\end{proof}


\subsection{Covering dimension of the flow space}
\label{sec:dimension-of-flow-space}

We will need the following elementary fact.

\begin{lemma}
  \label{lem:dim-overline-X}
  If $X$ is proper as a metric space and its covering dimension
  $\dim X$ is $\leq N$, then $\dim \overline{X} \leq N$.
\end{lemma}

\begin{proof}
  Let $\calu = \{U_i \mid i \in I\}$ be an open covering of $\overline{X}$.
  Recall from Remark~\ref{rem:cone-topology} that a basis 
  for the topology on $\overline{X}$ is given by sets of the form
  $\rho_r^{-1}(W)$ for $r \ge 0$ and open $W \subseteq \overline{B}_r(x_0)$,
  where we fix a base point $x_0$.)
  Thus for every $x \in \overline{X}$ there are $r_x$, 
  $W_x \subseteq \overline{B}_{r_x}(x_0)$ and $U_x \in \calu$ such that 
  $x \in \rho_{r_x}^{-1}(W_x) \subset U_x$.
  Because $\overline{X}$ is compact a finite number of the 
  sets $\rho_{r_x}^{-1}(W_x)$ cover $\overline{X}$.
  Note that $\rho_r = \rho_r|_{\overline{B}_{r'}(x_0)} \circ \rho_{r'}$,
  if $r' > r$.
  Therefore we can refine $\calu$ to a finite cover $\calv$,
  such that there is $r$ and a finite cover $\calw$ of $\overline{B_r(x_0)}$
  such that
  \begin{equation*}
    \calv = \rho_r^{-1}(\calw) := \{ \rho_r^{-1}(W) \mid W \in \calw \}.
  \end{equation*}
  The result follows because $\overline{B_r(x_0)}$ is closed in $X$
  and thus $\dim \overline{B_r(x_0)} \leq \dim X$.
\end{proof}

\begin{proposition}
  \label{pro:dim-for_FS-FSR}
  Assume that $X$ is proper and that $\dim X \leq N$.
  Then 
  $$\dim\bigl(\FS(X) - \FS(X)^\IR\bigr) \leq 3N+2.$$
\end{proposition}

\begin{proof}
  The image of any compact subset under a continuous map is compact and
  a bijective continuous map with a compact subset as source and Hausdorff
  space as target is a homeomorphism. 
  Hence every compact subset $K$ 
  of  $\FS(X) - \FS(X)^\IR$ is homeomorphic to a compact subset of
  $\overline{R} \x \overline{X} \x X \x \overline{X} \x \overline{R}$
  by Proposition~\ref{prop:embedding-of-FS} and
  hence its topological dimension satisfies because of
  Lemma~\ref{lem:dim-overline-X}
  Since $\FS(X)$ is a proper metric space by Lemma~\ref{prop:FS-is-proper},
  it is locally compact and can be written as 
  the countable union of compact subspaces
  and hence contains a countable dense subset. This implies
  that $\FS(X)$ has a countable basis for its topology.
  Since $\FS(X)- \FS(X)^{\IR}$ is an open subset of $\FS(X)$, 
  the topological space
  $\FS(X)- \FS(X)^{\IR}$ is locally compact and has a countable 
  basis for its topology.
  Now $\dim\bigl(\FS(X) - \FS(X)^\IR\bigr) \leq 3N+2$ follows 
  from~\cite[Exercise~9 in Chapter~7.9 on page~315]{Munkres(1975)}.
\end{proof}


\subsection{The flow space is locally connected}
\label{subsec:flow-space-locally-connected}

A topological space $Y$ is called \emph{semi-locally path-connected}
if for any $y \in Y$ and neighborhood $V$ of $y$ there is an open 
neighborhood $U$ of $y$ such that for every 
$z \in U$ there is a path $w$ in $V$ from $y$ to $z$.
Recall that $Y$ is called \emph{locally connected} or 
\emph{locally path-connected} if any neighborhood $V$ of any point 
$y \in Y$ contains an open neighborhood $U$ of $y$ such that $U$ itself is 
connected or path-connected respectively.
Suppose that $Y$ is semi-locally path-connected. 
Then any open subset of $Y$ is again semi-locally path-connected 
and each component of any open subset of $Y$ is an open subset
of $Y$. 
The latter is equivalent to the condition that $Y$ is locally connected.
Hence semi-locally path-connected implies locally connected.
The notion of semi-locally path-connected is 
weaker than the notion of locally
path-connected.

\begin{proposition}
  \label{prop:flow-space-is-locally-connected}
  Assume that $X$ is proper as a metric space.  
  Then $\FS(X) - \FS(X)^\IR$ is semi-locally path-connected.  
  In particular, $\FS(X) - \FS(X)^\IR$ is locally
  connected.
\end{proposition}

\begin{proof}
  Consider $c \in \FS(X) - \FS(X)^\IR$ and a neighborhood 
  $V \subseteq \FS(X) - \FS(X)^\IR$ of $c$. 
  By Lemma~\ref{lem:push-to-finite-geodesic} there
  is a homotopy $H_t \colon \FS(X) \to \FS(X)$ such
  that $H_0 = \id$ and $H_t(\FS(X)) \subseteq \FS(X)_f$ for all
  $t > 0$. 
  Lemma~\ref{lem:FS-IR-is-closed} implies that $V$ is also open 
  as a subset of $\FS(X)$.
  Since $H$ is continuous, there is $\delta > 0$ and an open 
  neighborhood $U_1 \subseteq V$ of $c$
  such that $H_t(U_1) \subseteq V$ for all $t \in [0,\delta]$.
  For any $d \in U_1$, $\omega_{d}(t) := H_{t\delta}(d)$ defines
  a path in $V$ from $d$ to $H_\delta(d)$.
  We have $H_\delta(c) \in \FS(X)_f$.
  From Proposition~\ref{prop:embedding-of-FS} we conclude that
  $\FS(X)_f - \FS(X)^\IR$ is open in
  $\FS(X) - \FS(X)^\IR$.
  By Proposition~\ref{prop:local-embeding-of-FS_f}
  we can find a path-connected neighborhood 
  $W \subseteq V \cap \left( \FS(X)_f - \FS(X)^\IR \right)$ 
  of $H_\delta(c)$.
  Set now $U := U_1 \cap (H_\delta)^{-1} (W)$.
  
  Consider $d \in U$.
  Then $H_\delta(c)$ and $H_\delta(d)$ both lie in $W$.
  Thus there is a path in $W$ from $H_\delta(c)$ to $H_\delta(d)$.
  This is in particular a path in $V$, since $W \subseteq V$. 
  Then $\omega := \omega_c \ast v \ast \omega_d$
  is a path in $V$ from $c$ to $d$.
  Hence $\FS(X) -  \FS(X)^{\IR}$ is semi-locally path connected.
\end{proof}


\subsection{The example of a complete Riemannian manifold 
      with non-positive sectional curvature}
  \label{subsec:example_non-pos-curv}

  Let $M$ be a simply connected complete Riemannian manifold 
  with non-positive sectional curvature.  
  It is a $\CAT(0)$-space with respect to the metric
  coming from the Riemannian metric (see~\cite[I.A.6 on
  page~173]{Bridson-Haefliger(1999)}).  
  Let $ST\!M$ be its sphere tangent bundle. 
  For every $x \in M$ and $v \in ST_xM$ there is precisely one geodesic
  $c_v \colon \IR \to M$ for which $c_v(0) = x$ and $c_v'(0) = v$ holds. 
  Given a geodesic $c \colon \IR \to M$ in $M$ and 
  $a_-,a_+ \in \overline{\IR}$ with
  $a_- \le a_+$, define the generalized geodesic 
  $c_{[a_-,a_+]} \colon \IR \to M$ 
  by sending $t$ to $c(a_-)$ if $t \le a_-$, to $c(t)$ if $a_- \le t \le
  a_+$, and to $c(a_+)$ if $t \ge a_+$. 
  Obviously $c_{[-\infty,\infty]} = c$.
  Let $d \colon \IR \to M$ be a generalized geodesic with $d_- < d_+$.  
  Then there is precisely one geodesic $\widehat{d} \colon \IR \to M$ with
  $\widehat{d}_{[d_-,d_+]} = d$.

  Define maps
  \begin{eqnarray*}
      \alpha \colon STM \times \bigl\{(a_i,a_+) \in \overline{\IR} \times 
      \overline{\IR}  \mid a_- < a_+\bigr\} \to \FS(M), 
      & \quad & 
      (v,a_i,a_+) \mapsto  c_v|_{[a_-,a_+]};
      \\
      \beta \colon \FS(M) \to STM \times \bigl\{(a_i,a_+) \in \overline{\IR}  \times 
      \overline{\IR}  \mid a_- < a_+\bigr\},
      & \quad &
      c \mapsto ({\widehat{c}\,}'(0),c_-,c_+).
    \end{eqnarray*}
  Then $\alpha$ and $\beta$ are to another inverse homeomorphisms. 
  They are compatible with the flow on $\FS(M)$ 
  of Definition~\ref{def_flow_space_FS(X)}, if one uses on 
  $STM \times  \bigl\{(a_i,a_+) \in \overline{\IR}  \times 
    \overline{\IR}  \mid a_- < a_+\bigr\}$ 
  the product flow given by the geodesic flow on $ST\!M$ and
  the flow on $\overline{R}$ which is at time $t$ given by the homeomorphism
  $\overline{\IR} \to \overline{\IR}$ sending $s \in \IR$ to $s-t$,
  $-\infty$ to $-\infty$, and $\infty$ to $\infty$.


\section{Dynamic properties of the flow space}
\label{sec:Contracting-transfers-FS}

\begin{summary*}
  In Definition~\ref{def:homotopy-action-on-B} introduce  
  the homotopy action that
  we will use to show that $\CAT(0)$-groups are transfer
  reducible over $\VCyc$. 
  It will act on a large ball in $X$.
  (The action of $G$ on the bordification $\overline{X}$
  is not suitable, because it has to large isotropy groups.)
  In Propositions~\ref{prop:flow-estimate-g}
  and~\ref{prop:flow-estimate-H} we study the dynamics
  of the flow with respect to the homotopy action.
  In the language of Section~\ref{sec:flow-spaces-and-S-long-covers} 
  this shows that $\FS(X)$ admits contracting transfers.
\end{summary*}

Throughout this section we fix the following convention.

\begin{convention}
  Let
  \begin{itemize}
  \item $(X,d_X)$ be a $\CAT(0)$-space which is proper
        as a metric space;
  \item $x_0 \in X$ be a fixed base point;
  \item $G$ be a group with a proper isometric action on $(X,d_X)$.  
  \end{itemize}
  For $x, y \in X$ and $t \in [0,1]$ we will denote
  by $t \cdot x + (1-t)\cdot y$ the unique point $z$ on the
  geodesic from $x$ to $y$ such that $d_X(x,z) = t d_X(x,y)$
  and $d_X(z,y) = (1-t) d_X(x,y)$.
  For $x,y \in X$ we will denote by $c_{x,y}$ the generalized
  geodesic determined by $(c_{x,y})_{-} = 0$, $c(-\infty) = x$
  and $c(\infty) = y$.
  (By~\cite[II.1.4(1)~on~p.160]{Bridson-Haefliger(1999)}
  and Lemma~\ref{prop:uniform-convergence-on-compact},
  $(x,y) \mapsto c_{x,y}$ defines a continuous map
  $X \x X \to \FS(X)$. 
  Note that $g \cdot c_{x,y} = c_{gx,gy}$.)
\end{convention}


\subsection{The homotopy action on $\overline{B}_R(x)$}
\label{subsec:homotopy-action-on-B}

Recall that for $r > 0$ and $z \in X$
we denote by $\rho_{r,z} \colon X \to \overline{B}_r(z)$
the canonical projection along geodesics, i.e.,
$\rho_{r,z}(x) = c_{z,x}(r)$, see also
Remark~\ref{rem:cone-topology}.
Note that $g \cdot \rho_{r,z}(x) = \rho_{r,gz}(gx)$
for $x,z \in X$ and $g \in G$.

\begin{definition}[The homotopy $S$-action on $\overline{B}_R(x_0)$]
  \label{def:homotopy-action-on-B}
  Let $S \subseteq G$ be a finite subset of $G$ with $e \in G$ and $R > 0$. 
  Define a homotopy $S$-action $(\varphi^R,H^R)$ on
  $\overline{B}_R(x)$ in the sense of 
  Definition~\ref{def:S-action_plus_long-covers}~%
\ref{def:S-action_plus_long-covers:action} as follows.  
  For $g \in S$, we define the map
  \[
    \varphi_g^R \colon \overline{B}_R(x_0) \to \overline{B}_R(x_0)
  \]
  by $\varphi^R_{g} (x) := \rho_{R,x_0} (gx)$.
  
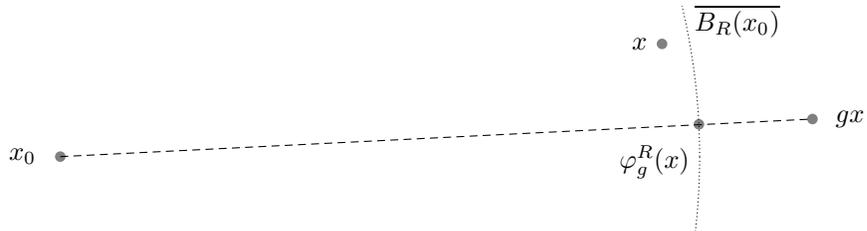
\begin{equation*}
    \begin{tikzpicture}
      \clip (-2,-1) rectangle (12,2);
      \fill [black,opacity=.5] (0,0) circle (2pt);
      \draw (-.5,0) node {$x_0$};
      \draw[densely dotted, thin] (0,0) circle (8.5);
      \fill [black,opacity=.5] (8,1.5) circle (2pt);
      \draw (7.7,1.5) node {$x$};
      \fill [black,opacity=.5] (10,.5) circle (2pt);
      \draw (10.5,.5) node {$gx$};
      \draw[thin, densely dashed] (0,0) -- (10,.5);
      \draw (9,1.8) node {$\overline{B_R(x_0)}$};
      \fill [black,opacity=.5] (8.49,.43) circle (2pt);
      \draw (7.9,-.07) node {$\varphi^R_g(x)$}; 
    \end{tikzpicture}
  \end{equation*}
  For $g,h \in S$ with $gh \in S$ we define the homotopy
  \[
    H^R_{g,h} \colon \varphi_g^R \circ \varphi_h^R \simeq \varphi_{gh}^R
  \] 
  by $H^R_{g,h}(x,t) := \rho_{R, x_0} 
         \bigl(t \cdot (ghx) + (1-t)\cdot (g \cdot \rho_{R,x_0}(hx))\bigr)$.
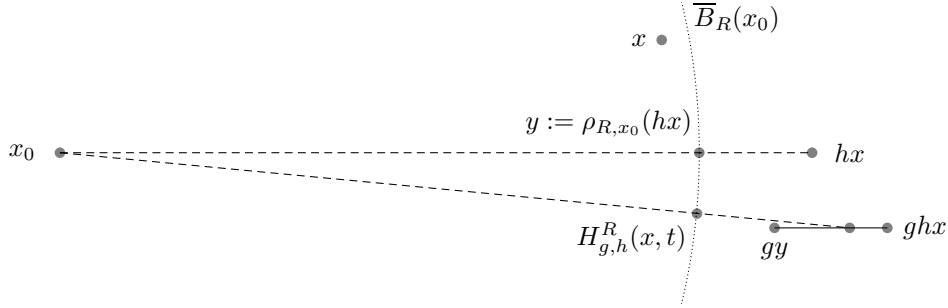
  \begin{equation*}
    \begin{tikzpicture}
      \useasboundingbox (2.3,-2.3) rectangle (9,2.3);
      \clip (-2,-2) rectangle (12,2);
      \fill [black,opacity=.5] (0,0) circle (2pt);
      \draw (-.5,0) node {$x_0$};
      \draw[densely dotted, thin] (0,0) circle (8.5);
      \fill [black,opacity=.5] (8,1.5) circle (2pt);
      \draw (7.7,1.5) node {$x$};
      \fill [black,opacity=.5] (10,0) circle (2pt);
      \draw (10.5,0) node {$hx$};
      \draw[thin, densely dashed] (0,0) -- (10,0);
      \draw (9,1.8) node {$\overline{B}_R(x_0)$};
      \fill [black,opacity=.5] (8.5,0) circle (2pt);
      \draw (7.3,.4) node {$y := \rho_{R,x_0}(hx)$};
      \fill [black,opacity=.5] (11,-1) circle (2pt);
      \draw (11.5,-1) node {$ghx$};
      \fill [black,opacity=.5] (9.5,-1) circle (2pt);
      \draw (9.5,-1.3) node {$gy$};
      \draw[thin] (9.5,-1) -- (11,-1);
      \fill [black,opacity=.5] (10.5,-1) circle (2pt);
      \draw [thin, densely dashed] (0,0) -- (10.5,-1);
      \fill [black,opacity=.5] (8.47,-.81) circle (2pt);
      \draw (7.6,-1.15) node {$H^R_{g,h}(x,t)$}; 
    \end{tikzpicture}
  \end{equation*}
\end{definition}

\begin{remark} 
  Notice that $H^R_{g,h}$ is indeed a homotopy from 
  $\varphi^R_g \circ \varphi^R_h$
  to $\varphi_{gh}$ since
  \begin{eqnarray*}
  H^R_{g,h}(x,0) 
  & = & 
  \rho_{R,x_0} 
     \bigl(0 \cdot (gh x) + 1 \cdot (g \cdot \rho_{R,x_0}(hx))\bigr)
  \\
  & = & 
  \rho_{R,x_0} 
     \bigl(g \cdot \rho_{R,x_0}(hx)\bigr)
  \\
  & = & 
  \varphi_g^R \circ \varphi_h^R(x),
  \end{eqnarray*}
  and
  \begin{eqnarray*}
  H^R_{g,h}(x,1) 
  & = & 
  \rho_{R,x_0} 
    \bigl(1 \cdot (ghx) + 0 \cdot (g \cdot \rho_{R,x_0}(hx)\bigr)
  \\ 
  & = & 
  \rho_{R,x_0}(ghx)
  \\ 
  & = & 
  \varphi_{gh}^R(x).
  \end{eqnarray*}
  It turns out that the more obvious homotopy given by convex combination
  $(x,t) \mapsto t \cdot \varphi_{gh}^R(x) + (1-t) \cdot \varphi_g^R \circ \varphi_h^R(x)$
  is not appropriate for our purposes.
\end{remark}

\begin{definition}[The map $\iota$]
  \label{def:the_mao_iota}
   Define the map 
   \[
     \iota \colon G \times X \to \FS(X)
   \]
   as follows. 
   For $(g,x) \in G \x X$ let $\iota(g,x) := c_{gx_0,gx}$.
\end{definition} 

The map $\iota$ is $G$-equivariant for the action an $G \x X$ defined
by $g \cdot (h,x) = (gh,x)$. 


\subsection{The flow estimate}
\label{subsec:flow-restimate}

\begin{proposition}
  \label{prop:flow-estimate-g}
  Let $\beta,L > 0$.
  For all $\delta > 0$ there are $T, r > 0$ such that
  for  $x_1,x_2 \in X$ with $d_X(x_1,x_2) \leq \beta$, 
  $x \in \overline{B}_{r+L}(x_1)$ there is 
  $\tau \in [-\beta,\beta]$ such that
  \begin{equation*}
    d_\FS\bigl( \Phi_{T}(c_{x_1,\rho_{r,x_1}(x)}), 
           \Phi_{T+\tau}(c_{x_2,\rho_{r,x_2}(x)}) \bigr) \leq \delta. 
  \end{equation*}
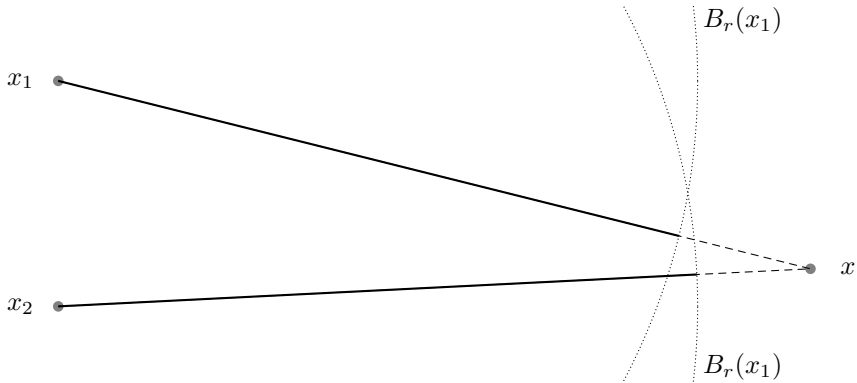
  \begin{equation*}
    \begin{tikzpicture}
      \clip (-2,-2) rectangle (12,3);
      \fill [black,opacity=.5] (0,-1) circle (2pt);
      \draw (-.5,-1) node {$x_2$};
      \fill [black,opacity=.5] (0,2) circle (2pt);
      \draw (-.5,2) node {$x_1$};
      \fill [black,opacity=.5] (10,-.5) circle (2pt);
      \draw (10.5,-.5) node {$x$};
      \draw[densely dotted, thin] (0,-1) circle (8.5);
      \draw (9.1,2.8) node {$B_r(x_1)$};
      \draw[densely dotted, thin] (0,2) circle (8.5); 
      \draw (9.1,-1.8) node {$B_r(x_1)$};
      \draw[thin, densely dashed] (0,2) -- (10,-.5); 
      \draw[thin, densely dashed] (0,-1) -- (10,-.5);
      \begin{scope}
        \clip (0,-1) circle (8.5);
        \draw[thick] (0,-1) -- (10,-.5);
      \end{scope} 
      \begin{scope}
        \clip (0,2) circle (8.5);
        \draw[thick] (0,2) -- (10,-.5);
      \end{scope}
    \end{tikzpicture}
  \end{equation*}
\end{proposition}

The proof requires some preparation.

\begin{lemma}
  \label{lem:triangles}
  Let $r', L, \beta > 0$, $r'' > \beta$.
  Set $T := r'' + r'$, $r := r'' + 2r' + \beta$.
  Let $x_1, x_2 \in X$ such that $d_X(x_1,x_2) \leq \beta$.
  Let $x \in \overline{B}_{r+L}(x_1)$.
  Set $\tau := d_X(x_2,x) - d_X(x_1,x)$.
  Then for all $t \in [T - r', T + r']$
  \begin{enumerate}
  \item \label{lem:triangles:estimate} 
       $d_X(c_{x_1,x}(t),c_{x_2,x}(t + \tau)) \leq
             \frac{2\cdot \beta \cdot   (L + 2r' + \beta)}{r''}$; 
  \item \label{lem:triangles:no-rho}
    $c_{x_1,\rho_{r,x_1}(x)}(t) = c_{x_1,x}(t)$ and
    $c_{x_2,\rho_{r,x_2}(x)}(t+\tau) = c_{x_2,x}(t+\tau)$.  
  \end{enumerate}
\end{lemma}

\begin{proof}~\ref{lem:triangles:estimate}
  Let $t \in [T - r', T + r']$.
  Note that $|\tau| \leq \beta$. 
  From $T - r' = r'' > \beta$ we conclude $t, t + \tau > 0$.
  If $t \geq d_X(x,x_1)$, then 
  $c_{x_1,x}(t) = x = c_{x_2,x}(t+\tau)$
  and the assertion follows
  in this case. Hence we can assume $0 < t < d_X(x,x_1)$.
  A straight forward computation shows that 
  $0 < t + \tau < d_X(x,x_2)$ and 
  $d_X(c_{x_1,x}(t),x) = d_X(c_{x_2,x}(t + \tau),x)$.
  We get $r'' = T - r' \leq t < d_X(x,x_1)$.
  Applying the $\CAT(0)$-condition to the 
  triangle $\Delta_{x,x_1,x_2}$ we deduce
  that
  \begin{equation*}
    d_X(c_{x_1,x}(t),c_{x_2,x}(t + \tau)) \leq
    \frac{2 \cdot d_X(x_1,x_2) \cdot  (d_X(x,x_1) - t)}{d_X(x,x_1)}
     \leq \frac{2 \cdot \beta \cdot (d_X(x,x_1) - t)}{d_X(x,x_1)}.
  \end{equation*}
  Combining this with $d_X(x,x_1) > r''$ and 
  $d_X(x,x_1) - t \leq (r + L) - (T - r') 
       = r'' + 2r' + \beta + L - r'' - r' + r' 
       = 2r' + \beta + L$
  we obtain the asserted inequality. 
  \\[1ex]~\ref{lem:triangles:no-rho}
  We have $t \leq T + r' = 2r' + r'' = r - \beta$
  and $t \geq T - r' = r'' > \beta$.
  Thus $t,t + \tau \in [0,r]$.
  Thus 
  $c_{x_1,\rho_{r,x_1}(x)}(t) = c_{x_1,x}(t)$ and
  $c_{x_2,\rho_{r,x_2}(x)}(t+\tau) = c_{x_2,x}(t+\tau)$.  
\end{proof}

\begin{proof}[Proof of Proposition~\ref{prop:flow-estimate-g}]
  Let $\delta > 0$ be given.
  Pick $r' > 0$, $r'' > \beta$, $1 > \delta' > 0$ such that
  \begin{equation*}
    \int_{-\infty}^{-r'} \frac{1 + |t|}{e^{|t|}}dt 
    \leq  \frac{\delta}{3} 
    \quad , \quad
    \int_{-r'}^{r'} \frac{\delta'}{2e^{|t|}} dt \leq \frac{\delta}{3} 
  \end{equation*}
  and
  \begin{equation*}
    \frac{2 \cdot \beta  (L + 2r' + \beta)}{r''}
    \leq 
    \delta'.
  \end{equation*}
  Set $r := 2r' + r'' + \beta$ and 
  $T := r' + r''$.
  Let $x_1, x_2 \in X$ with $d_X(x_1,x_2) \leq \beta$.
  Let $x \in \overline{B}_{r}(x_1)$ be given.
  Set $\tau := d_X(x_2,x) - d_X(x_1,x)$.
  Then $|\tau| \leq d_X(x_2,x_1) \leq \beta$.    
  Using Lemma~\ref{lem:triangles} we conclude that
  for all $t \in [-r',r']$
  \begin{eqnarray*}
    \lefteqn{ d_X \bigl( c_{x_1,\rho_{r,x_1}(x)} (T + t), 
                    c_{x_2,\rho_{r,x_2}(x)} (T + t + \tau) \bigr) 
    } \quad \quad \quad \quad
    & & \\
    & = & 
    d_X \bigl( c_{x_1,x} (T + t),  c_{x_2,x} (T + t + \tau) \bigr) 
    \\
    & \leq & \frac{2 \cdot  \beta  (L + 2r' + \beta)}{r''}
    \; \leq \; \delta'.
  \end{eqnarray*}
  Thus
  \begin{eqnarray*}
    \lefteqn{ d_\FS\bigl( \Phi_{T}(c_{x_1,\rho_{r,x_1}(x)}), 
                      \Phi_{T+\tau}(c_{x_2,\rho_{r,x_2}(x)}) \bigr)
    } \quad \quad
    & & \\
    & = &
    \int_{-\infty}^{\infty} 
       \frac{ d_X \bigl( c_{x_1,\rho_{r,x_1}(x)} (T + t), 
                    c_{x_2,\rho_{r,x_2}(x)} (T + t + \tau) \bigr) }
            { 2 e^{|t|} } 
    dt
    \\
    & \leq &
    \int_{-\infty}^{-r'} 
       \frac{ 2|t| + \delta'}{ 2 e^{|t|} } dt
    + \int_{-r'}^{r'} 
       \frac{ \delta'}{ 2 e^{|t|} } dt
    + \int_{r'}^{\infty} 
       \frac{ \delta'  + 2|t| }{ 2 e^{|t|} } dt 
    \\
    & \leq &
    \int_{-\infty}^{-r'} 
       \frac{ |t| + 1}{  e^{|t|} } dt
    + \int_{-r'}^{r'} 
       \frac{ \delta'}{ 2 e^{|t|} } dt
    + \int_{r'}^{\infty} 
       \frac{ 1  + |t| }{  e^{|t|} } dt 
    \\
    & \leq &
    \frac{\delta}{3} + \frac{\delta}{3} + \frac{\delta}{3} 
    \; = \; \delta.
  \end{eqnarray*}
\end{proof}

\begin{lemma}
  \label{lem:flow-distorsion}
  Let $\e > 0$, $\beta > 0$. 
  Then there is $\delta > 0$ such that for all $|\tau| \leq \beta$
  \begin{equation*}
    d_\FS(c_0,c_1) \leq \delta \; \implies 
             d_\FS\bigl(\Phi_\tau(c_0),\Phi_\tau(c_1)\bigr) \leq \e
  \end{equation*}
  for $c_0$, $c_1 \in \FS(X)$.
\end{lemma}

\begin{proof}
  This follows from Lemma~\ref{lem:Phi_well-defined}.
\end{proof}

\begin{proposition}
  \label{prop:flow-estimate-H}
  Let $S$ be a finite subset of $G$ (containing $e$).
  Then there is $\beta >0$ such that the following holds:
 
  For all $\delta > 0$ there are $T,R > 0$ such that
  for every $(a,x) \in G \x \overline{B}_R(X)$, 
  $s \in S$, $f \in F_s(\varphi^R,H^R)$
  there is $\tau \in [-\beta,\beta]$ such that 
  \begin{equation*}
      d_\FS\left(\Phi_T(\iota(a,x)),
          \Phi_{T+\tau}(\iota(as^{-1},f(x)))\right) \leq \delta.
  \end{equation*}
\end{proposition}

\begin{proof}
  Pick $\beta$ such that $\frac{\beta}{2} \geq d_X(s x_0, x_0)$ 
  for all $s \in S$.
  Let $L := \beta$.
  Let $\delta > 0$ be given.
  By Lemma~\ref{lem:flow-distorsion} there is 
  $\frac{\delta}{2} > \delta_0 > 0$ such that
  for $|\tau'| \leq \beta$ 
  \begin{eqnarray}
    d_\FS(c_0,c_1) \leq \delta_0 \; \implies 
             d_\FS\bigl(\Phi_{\tau'}(c_0),\Phi_{\tau'}(c_1)\bigr) 
         \leq \frac{\delta}{2}
  \label{eq:ottl_und_lell}
  \end{eqnarray}
  for $c_0$, $c_1 \in \FS(X)$.
  By Proposition~\ref{prop:flow-estimate-g} there are $T, R > 0$
  such that for $x,x_1,x_2 \in X$ with $d_X(x_1,x_2) \leq \beta$
  and $d_X(x,x_1) \leq R + L$ there is 
  $\tau = \tau(x,x_1,x_2) \in [-\frac{\beta}{2},\frac{\beta}{2}]$
  such that 
  \begin{equation*}
    d_\FS \big( \Phi_{T}(c_{x_1,\rho_{R,x_1}(x)}), 
           \Phi_{T+\tau}(c_{x_2,\rho_{R,x_2}(x)}) \bigr) \leq \delta_0. 
  \end{equation*}
  Let $(a,x) \in G \in \overline{B}_R(X)$, 
  $s \in S$, $f \in F_s(\varphi^R,H^R)$.
  Note that 
  \begin{equation*}
    d_\FS\bigl(\Phi_T(\iota(a,x)),
          \Phi_{T+\tau}(\iota(as^{-1},f(x)))\bigr) 
    = 
    d_\FS\bigl(\Phi_T(\iota(e,x)),
          \Phi_{T+\tau}(\iota(s^{-1},f(x)))\bigr) 
  \end{equation*}
  because $\iota$ and $\Phi$ are $G$-equivariant
  and $d_{\FS}$ is $G$-invariant.
  Thus it suffices to consider the case $a = e$.
  Then there are $t \in [0,1]$ and $g,h \in S$ such that
  $s = gh$ and 
  $f(x) = H^R_{g,h}(x,t) = \rho_{R,x_0} 
         \bigl(t \cdot (ghx) + (1-t)\cdot (g \cdot
         \rho_{R,x_0}(hx))\bigr)$.
  Therefore 
  $s^{-1} f(x) = \rho_{R,s^{-1}x_0} 
       \bigl(t \cdot x + (1-t) \cdot \rho_{R,h^{-1}x_0}(x) \bigr)$.
  Set $z := t \cdot x + (1-t) \cdot \rho_{R,h^{-1}x_0}(x)$.
  Then $\iota(s^{-1},f(x)) = c_{s^{-1}x_0,\rho_{R,s^{-1}x_0}(z)}$.
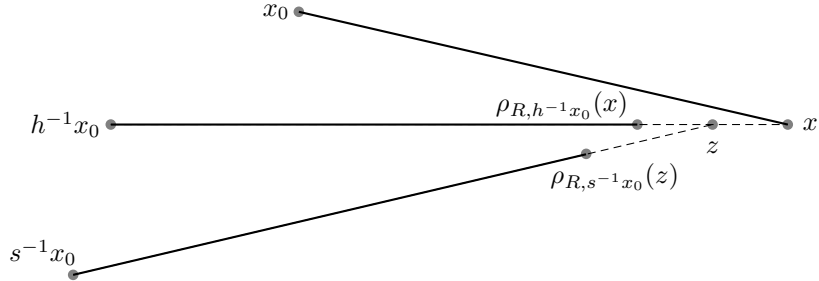
  \begin{equation*}
    \begin{tikzpicture}
       \clip (-2,-2.2) rectangle (10,1.8);
       \fill [black,opacity=.5] (2.5,1.5) circle (2pt);
       \draw (2.2,1.5) node {$x_0$};
       \fill [black,opacity=.5] (9,0) circle (2pt);
       \draw (9.3,0) node {$x$};
       \fill [black,opacity=.5] (0,0) circle (2pt);
       \draw (-.6,0) node {$h^{-1}x_0$}; 
       \fill [black,opacity=.5] (7,0) circle (2pt);
       \draw (6,.25) node {$\rho_{R,h^{-1}x_0}(x)$};
       \fill [black,opacity=.5] (8,0) circle (2pt);
       \draw (8,-.3) node {$z$};
       \fill [black,opacity=.5] (-.5,-2) circle (2pt);
       \draw (-.9,-1.7) node {$s^{-1}x_0$};
       \fill [black,opacity=.5] (6.32,-.39) circle (2pt);
       \draw (6.7,-.7) node {$\rho_{R,s^{-1}x_0}(z)$};
       \draw[thick] (2.5,1.5) -- (9,0);
       \draw[densely dashed, thin] (0,0) -- (9,0);
       \draw[thick] (0,0) -- (7,0); 
       \draw[densely dashed, thin] (-.5,-2) -- (8,0);
       \begin{scope}
         \clip (-.5,-2) circle (7);
         \draw[thick] (-.5,-2) -- (8,0);
       \end{scope}
    \end{tikzpicture}
  \end{equation*} 
  We have $d_X(x,x_0) \leq R$.
  Moreover, $d_X(z,h^{-1}x_0) \leq d_X(x,h^{-1}x_0) 
                    \leq d_X(x,x_0) + d_X(x_0,h^{-1}x_0) 
                    \leq R + L$. 
  Therefore, we can set 
  $\tau_1 := \tau(x,x_0,h^{-1}x_0)$,
  $\tau_2 := \tau(z,h^{-1}x_0,s^{-1}x_0)$ and
  $\tau := \tau_1 + \tau_2$.
  Note that $|\tau| \leq \beta$, since $|\tau_i| \leq \frac{\beta}{2}$.
  We have
  \begin{equation*}
        d_\FS \bigl( \Phi_{T}(c_{h^{-1}x_0,\rho_{R,h^{-1}x_0}(z)}) , 
                \Phi_{T+\tau_2}(c_{s^{-1}x_0,\rho_{R,s^{-1}x_0}(z)})  \bigr)
        \leq \delta_0
  \end{equation*}
  and therefore by~\eqref{eq:ottl_und_lell}
  \begin{equation*}
         d_\FS \bigl( \Phi_{T+\tau_1}(c_{h^{-1}x_0,\rho_{R,h^{-1}x_0}(z)}) , 
                \Phi_{T+\tau_1+\tau_2}(c_{s^{-1}x_0,\rho_{R,s^{-1}x_0}(z)})  \bigr)
        \leq \frac{\delta}{2}.
  \end{equation*}
  Thus
  \begin{eqnarray*}
    \lefteqn{ d_\FS \bigl( \Phi_T(\iota(e,x)),
                    \Phi_{T+\tau}(\iota(s^{-1},f(x))) \bigr) }
   & & \\
   & = & d_\FS \bigl( \Phi_T(c_{x_0,x}), 
                \Phi_{T+\tau}(c_{s^{-1}x_0,\rho_{R,s^{-1}x_0}(z)})  \bigr)
   \\
   & \leq & d_\FS \bigl( \Phi_T(c_{x_0,x}), 
                \Phi_{T+\tau_1}(c_{h^{-1}x_0,\rho_{R,h^{-1}x_0}(x)})  \bigr)
   \\ & &
           + \; d_\FS \bigl( \Phi_{T+\tau_1}(c_{h^{-1}x_0,\rho_{R,h^{-1}x_0}(z)}) , 
                \Phi_{T+\tau_1+\tau_2}(c_{s^{-1}x_0,\rho_{R,s^{-1}x_0}(z)})  \bigr)
  \\
  & \leq & \frac{\delta}{2} + \frac{\delta}{2} = \delta.
  \end{eqnarray*}
\end{proof}


\typeout{---------- Orbits with bounded G-period ----------------}

\section{Orbits with bounded $G$-period}
\label{sec:bounded-period}

\begin{summary*}
  Let $\FS(X)_{\le \gamma}$ be the part of $\FS(X)$ that contains of in 
  some sense
  periodic orbits, namely, those generalized geodesics for which 
  there exists for
  every $\epsilon > 0$ an element $\tau \in (0, \gamma + \epsilon]$ and 
  $g \in G$ such that $g\cdot c = \Phi_{\tau}(c)$ holds
  (see~\eqref{FS_less_or_equal_gamma}).
  Our main result here is Theorem~\ref{the:covering_of_FS(X)_gamma} 
  that asserts 
  that there is a cover of uniformly bounded dimension for 
  $\FS(X)_{\leq \gamma}$
  that is long in the direction of the flow.
  To this end we study hyperbolic elements in $G$ and their axis.
  These come in parallel families (called $\FS_a$ below) 
  that project to convex subspaces of $X$.
  We construct the desired cover first for the $\FS_a$ by 
  considering the quotient $Y_a$ of $\FS_a$ by the flow.
  One difficulty here is that the group $G_a$ that naturally 
  acts on $Y_a$, does so with infinite isotropy. 
  The isotropy groups here are virtually cyclic and this forces
  the appearance of the family  $\VCyc$ in 
  Theorem~\ref{the:covering_of_FS(X)_gamma} and in our main result.
\end{summary*}

Throughout this section we fix the following convention.

\begin{convention} \label{con:(X,G,K)}
  Let
  \begin{itemize}
  \item $(X,d_X)$ be a $\CAT(0)$-space which is proper
        as a metric space and has finite covering dimension;
  \item $G$ be a group with a proper  
        isometric action on $(X,d_X)$;
  \item $K \subseteq X$ be a compact subset.  
  \end{itemize}
\end{convention}

The following is the main result of this section.

\begin{theorem} 
  \label{the:covering_of_FS(X)_gamma}
  There is a natural number $M$  such that for every 
  $\gamma > 0$ there exists a collection $\calv$ of subsets 
  of $\FS(X)$ satisfying:
  \begin{enumerate}
  \item \label{the:covering_of_FS(X)_gamma:VCyc}
      Each element $V \in \calv$ is an open
      $\VCyc$-subsets of the $G$-space $\FS(X)$  
      (see Definition~\ref{def:F-cover});
  \item \label{the:covering_of_FS(X)_gamma:invariance}
      $\calv$ is $G$-invariant; i.e., for $g \in G$ and 
      $V\in \calv$ we have $g\cdot V \in \calv$;
  \item \label{the:covering_of_FS(X)_gamma:finite}
      $G\backslash \calv$ is finite; 
  \item \label{the:covering_of_FS(X)_gamma:dim}
      We have $\dim \calv \leq M$;
  \item \label{the:covering_of_FS(X)_gamma:flow}
      There is $\e > 0$ with the following property:
      for $c \in \FS_{\leq \gamma}$ such that $c(t) \in G \cdot K$ for
      some $t \in \IR$ there is $V \in \calv$ such
      that $B_\e(\Phi_{[-\gamma,\gamma]}(c)) \subseteq V$.
\end{enumerate}
\end{theorem}


\subsection{Hyperbolic isometries of spaces}
\label{subsec:hyperbolic-isometries}

We recall some basic facts about isometries of a $\CAT(0)$-space
from~\cite[Chapter~II.6]{Bridson-Haefliger(1999)}. 
Let $\gamma \colon X \to X$  be an isometry. 
The \emph{displacement function} of $\gamma$ is defined by
\begin{eqnarray*}
  d_{\gamma} \colon  X & \to & [0,\infty), \quad x \mapsto d_X(\gamma x,x).
\end{eqnarray*}
The \emph{translation length} of $\gamma$ is defined  by 
\begin{eqnarray*}
  l(\gamma ) & := & \inf\bigl\{d_{\gamma}(x) \mid x \in X\bigr\}.
\end{eqnarray*}
Define a subspace of $X$ by 
\begin{eqnarray*}
  \Min(\gamma) & := & \bigl\{x \in X\mid d_{\gamma}(x) = l(\gamma )\bigr\}.
\end{eqnarray*}
We call $\gamma$ \emph{elliptic} if $\gamma$ has a fixed point and
\emph{hyperbolic} if the displacement function $d_{\gamma}$ attains a 
strictly positive minimum, or, equivalently, 
$l(\gamma ) > 0$ and $\Min(\gamma) \not= \emptyset$. 

\begin{lemma}
  \label{lem:basic_prop_l_plus_Min}
  Let $\gamma \colon X \to X$ be an isometry.
  \begin{enumerate}
  \item \label{lem:basic_prop_l_plus_Min:conjugation}
      If $\alpha \colon X \to X$ is an isometry, 
      then $l(\gamma ) = l(\alpha\gamma\alpha^{-1})$ 
      and $\Min(\alpha\gamma\alpha^{-1}) = \alpha(\Min(\gamma))$;
\label{item:1}  \item \label{lem:basic_prop_l_plus_Min:closed_plus_convex}
      $\Min(\gamma)$ is a closed convex set;
  \item \label{lem:basic_prop_l_plus_Min:hyperbolic}
      The isometry $\gamma$ is hyperbolic  if and only if it
      possesses an \emph{axis}, i.e., there is a geodesic 
      $c \colon \IR \to X$ and $\tau > 0$
      such that $\gamma \circ c(t) = c(t + \tau)$ holds for $t \in \IR$.
      In this case $\tau = l(\gamma )$;
  \item \label{lem:basic_prop_l_plus_Min:union_of_axis}
      Two axes $c$ and $d$ for the hyperbolic isometry $\gamma$ 
      are \emph{parallel}, i.e., 
      $d_X(c(t),d(t))$ is constant in $t \in \IR$. 
      The union of the images $c(\IR)$ of all
      axes $c$ for $\gamma$ is $\Min(\gamma)$.
  \end{enumerate}
\end{lemma}

\begin{proof}
  See~\cite[II.6.2(2) on p.229]{Bridson-Haefliger(1999)} 
  for~\ref{lem:basic_prop_l_plus_Min:conjugation}. 
  See~\cite[II.6.2(3) on p.229]{Bridson-Haefliger(1999)}
  for~\ref{lem:basic_prop_l_plus_Min:closed_plus_convex}.
  See~\cite[II.6.8(1) on p.231]{Bridson-Haefliger(1999)}
  for~\ref{lem:basic_prop_l_plus_Min:hyperbolic}. 
  See~\cite[II.6.8(3) on p.231]{Bridson-Haefliger(1999)} 
  for~\ref{lem:basic_prop_l_plus_Min:union_of_axis}.
\end{proof}

We emphasize that an axis for a hyperbolic element $\gamma$ is a 
(parametrized) geodesic $c \colon \IR \to X$ and not only $c(\IR)$. 
So two hyperbolic elements $\gamma_1$ and $\gamma_2$ have a common axis 
if there exists a  geodesic $c \colon \IR \to X$ such that 
$\gamma_1\cdot c(t) = c(t + l(\gamma_1))$ and
$\gamma_2\cdot c(t) = c(t + l(\gamma_2))$ holds for all $t \in \IR$.
We denote by $l(g)$ the 
translation length of the isometry $X \to X$ given by 
multiplication with $g \in G$. 
We will say that $g$ is hyperbolic if this isometry is hyperbolic.

\begin{lemma}
  \label{lem:type-I-VCyc}
  Let $c\colon \IR \to X$ be a geodesic. 
  Put 
  \begin{multline*}
  \hspace{10mm} G_{\Phi_{\IR}(c)} = \{g \in G \mid g(\Phi_{\IR}(c)) 
                                = \Phi_{\IR}(c)\} 
  \\
  =
  \bigl\{g \in G \mid \exists \tau \in \IR \;\text{with}\; 
             gc(t) = c(t + \tau)  \;\text{for all} \; t \in \IR\bigr\}
  \end{multline*}
  and 
  $$G_{c(\IR)} := \bigl\{g \in G \mid g \cdot c(\IR) = c(\IR)\}.$$
  Then $G_{\Phi_{\IR}(c)}$ is virtually cyclic of type I and
  $G_{c(\IR)}$ is virtually cyclic.
\end{lemma}    

\begin{proof}
  The group $G_{c(\IR)}$ acts properly and isometrically on $\IR$ since
  $c(\IR)$ is isometric to $\IR$. 
  The isometry group of $\IR$ fits into the exact sequence
  $1 \to \IR \xrightarrow{i} \Isom(\IR) \xrightarrow{p}  \{\pm 1\}\to 1$, 
  where $i$ sends a real number $r$ to the isometry $t \mapsto t + r$ and
  $p$ sends an isometry to $1$ if it is strictly monotone increasing 
  and to $-1$ otherwise.
  Since $G_{c(\IR)}$ acts properly on $\IR$, the obvious homomorphism
  $G \to \Isom(\IR)$ has a finite kernel and its image is a 
  discrete subgroup of $\Isom(\IR)$. 
  This implies that $G_{c(\IR)}$ is virtually cyclic.  

  Since the action of $G_{\Phi_{\IR}(c)}$ on $c$ is by translations, the 
  same argument shows that $G_{\Phi_{\IR}(c)}$ is virtually cyclic of type I.
\end{proof}


\subsection{Axes in the flow space}
\label{subsec:axis-in-FS}

\begin{notation}
  \label{not:A_gamma}
  Let $\gamma > 0$. Let $K$ be the compact subset
      from Convention~\ref{con:(X,G,K)}.
  \begin{enumerate}
  \item 
      Let 
      $$G^{\hyp}_{\leq \gamma} \subseteq G$$
      be the set
      of all hyperbolic $g \in G$ of translation length $l(g) \leq \gamma$
      such that some axis $c$ for $g$
      intersects $G \cdot K$. 
      
      Consider the equivalence relation $\sim$ on $G^{\hyp}_{\leq \gamma}$
      for which $g \sim g'$ if and only if there exists parallel axes
      $c_g$ and $c_{g'}$ for $g$ and $g'$.
      (This relation is transitive by 
      Lemma~\ref{lem:basic_prop_l_plus_Min}~\ref{lem:basic_prop_l_plus_Min:union_of_axis}.)
      Put
      $$A_{\leq \gamma} := G^{\hyp}_{\leq \gamma} / \sim.$$
      The conjugation action of $G$ on $G$ restricts to
      an action on $G^{\hyp}_{\leq \gamma}$ and
      descends to an action of $G$ on $A_{\leq \gamma}$, see 
      Lemma~\ref{lem:basic_prop_l_plus_Min}~\ref{lem:basic_prop_l_plus_Min:conjugation}.
      For $a \in A_{\leq \gamma}$ we set 
      $$G_a := \{ g \in G \mid g \cdot a = a \}.$$  
  \item 
      For $a \in A_{\leq \gamma}$ let 
      $$\FS_a \subseteq \FS(X)$$ 
      denote the subspace of
      $\FS(X)$ that consists of all  geodesics 
      $c \colon \IR \to X$, that are an axis for some $g \in a$
      and intersect $G \cdot K$.
      We remark that $c \in \FS(X)$ is an axis for $g$ if and only
      if $\Phi_\tau(c) = gc$ for some $\tau > 0$ and in this case
      $\tau = l(g)$, see 
      Lemma~\ref{lem:basic_prop_l_plus_Min}~\ref{lem:basic_prop_l_plus_Min:hyperbolic}.
      Define 
      $$p_a \colon \FS_a \to X, \quad c \mapsto c(0).$$
      We denote by 
      $$Y_a := \FS_a/\Phi$$ 
      the quotient
      of $\FS_a$ by the action of the flow $\Phi$. Let
      $$q_a \colon \FS_a \to Y_a$$ 
      be the canonical projection.
      The action of $G$ on $\FS(X)$ restricts to an action of
      $G_a$ on $\FS_a$.
      Because $p_a$ is $G_a$-equivariant and because the flow $\Phi$
      commutes with the $G$-action on $\FS(X)$ we obtain an action
      of $G_a$ on $Y_a$ and $q_a$ is $G_a$-equivariant for this action.  
  \end{enumerate}
\end{notation}

\begin{lemma}
  \label{lem:proper-action-constant-g_n}
  Let $(Z,d_Z)$ be a proper metric space 
  with a proper isometric action of a group $H$.
  If $(z_n)_{n \in \IN}$ and $(h_n)_{n \in \IN}$ are sequences
  in $Z$ and $H$ such that $z_n \to z$ and $h_nz_n \to z'$
  converge in $Z$, then $\{h_n \mid n \in \IN \}$ is finite
  and for every $h \in H$ such that $h_n = h$ for infinitely many $n$
  we have $hz = z'$.
\end{lemma}

\begin{proof}
  Let $n_0 > 0$ such that $d_Z(z,z_n) < 1$ and $d_Z(z',h_nz_n) < 1$
  for all $n \geq n_0$.
  Thus $d_Z(h_nz,z') < 2$ for all $n \geq n_0$.
  Thus $d_Z((h_{n_0})^{-1}h_n z,z) < 4$ for all $n \geq n_0$.
  Thus $\{ (h_{n_0})^{-1}h_n \mid n \geq n_0 \}$ is finite,
  because the action is proper.
  Therefore $\{h_n \mid n \in \IN \}$ is finite.  
  If $h_n = h$ for infinitely many $n \in \IN$, then
  $hz = \lim_{n \to \infty} hz_n = 
             \lim_{n \to \infty} h_nz_n = z'$.
\end{proof}

\begin{corollary}
  \label{cor:G-cdot-L-closed}
  Let $(Z,d_Z)$ be a proper metric space 
  with a proper isometric action of a group $H$.
  If $L \subset Z$ is compact, then $H \cdot L \subset Z$
  is closed. 
\end{corollary}

\begin{proof}
  Let $h_n z_n \to z$ with $h_n \in H$ and $z_n \in L$.
  After passing to a subsequence we have $z_n \to z'$.
  Lemma~\ref{lem:proper-action-constant-g_n} implies that we
  can pass to a further subsequence for which $h_n = h$ is constant.
  Thus $z \in h\cdot L \subset H \cdot L$. 
\end{proof}

\begin{lemma}
  \label{lem:K_gamma}
  There is a compact subset $K_\gamma \subseteq X$ such that
  $c(0) \in G \cdot K_\gamma$ for all $c \in \FS_a$, $a \in A_{\leq \gamma}$.  
\end{lemma}

\begin{proof}
  Set $K_\gamma := \overline{B}_\gamma(K)$.
  This is compact because $K$ is compact and $X$ is proper as
  a metric space.
  For $a \in A_{\leq \gamma}$, $c \in \FS_a$ there are $g \in a$,
  $t \in [0,\gamma]$ such that $c(s + t) = g c(s)$ for all $s \in \IR$.
  Because $c$ intersects $G \cdot K$, this implies that there is
  $s_0 \in [0,\gamma]$ such that $c(s_0) \in G \cdot K$.
  Thus $c(0) \in G \cdot \overline{B}_\gamma(K) = G \cdot K_\gamma$.  
\end{proof}

\begin{lemma}
  \label{lem:seq-in-FS}
  Let $\gamma > 0$.
  Let $(c_n)_{n \in \IN}$ be a sequence in 
  $\bigcup_{a \in A_{\leq \gamma}} \FS_a$ that converges to $c$ in 
  $\FS(X)$.
  Then there are $g \in G^{\hyp}_{\leq \gamma}$ and
  an infinite subset $I \subseteq \IN$ such that
  $c$ and all $c_i$, $i \in I$ are axes for $g$ and 
  intersect $G \cdot K$, where $K$ is the compact subset
  from Convention~\ref{con:(X,G,K)}.
  In particular, $c \in \FS_a$ and $c_{i} \in \FS_a$ for all $i \in I$, 
  where $a \in A_{\leq \gamma}$ is the class of $g$.
\end{lemma}

\begin{proof}
  There are $g_n \in G$, $t_n \in [0,\gamma]$ 
  such that $g_nc_n = \Phi_{t_n}(c_n)$.
  We can pass to a subsequence and assume that $t_n \to t_0$.
  Then $g_nc_n = \Phi_{t_n}(c_n) \to \Phi_{t_0}(c)$ (using
  Lemma~\ref{lem:Phi_well-defined}).
  Since $\FS(X)$ is proper and $G$ acts properly on $\FS(X)$
  (see Propositions~\ref{prop:FS-is-proper} 
  and~\ref{prop:cocompact})
  we can apply Lemma~\ref{lem:proper-action-constant-g_n}
  and assume after passing to a further subsequence
  that $g_n = g$ is constant.
  Then $gc = \lim g_n c_n = \lim \Phi_{t_n} c_n = \Phi_{t_0}(c)$.
  It remains to show that $c$ intersects $G \cdot K$.
  
  For each $n$ there is $s_n$ in $\IR$ such that $c_n(s_n) \in G \cdot K$.
  Since $g_n c_n(s) = c_n(s + t_n)$ and  $t_n \in [0,\gamma]$, we can arrange
  $s_n \in [0,\gamma]$ for all $n \ge 0$. 
  By passing to a subsequence we can arrange 
  that $s_n \to s_0$ for $n \to \infty$ for some $s_0 \in [0,\gamma]$.
  Then $c_n(s_n) \to c(s_0)$ 
  (using Proposition~\ref{prop:uniform-convergence-on-compact}).
  By Corollary~\ref{cor:G-cdot-L-closed} $G \cdot K$ is closed.
  Thus $c(s) \in G \cdot K$. 
\end{proof}

\begin{lemma}
  \label{lem:properties-A}
  Let $\gamma > 0$. Then 
  \begin{enumerate}
  \item \label{lem:properties-A:cofinite}
      $G \backslash A_{\leq \gamma}$ is finite;
  \item \label{lem:properties-A:GFS_a-closed}
      $G \cdot \FS_a \subseteq \FS(X)$
      is closed for all $a \in A_{\leq \gamma}$; 
      there is $K_a \subseteq G \cdot \FS_a$ compact
      such that $G \cdot K_a = G \cdot \FS_a$; 
  \item \label{lem:properties-A:a-neq-b} 
      there is $\e > 0$ such that $d_X(\FS_a,\FS_b) > \e$ for 
      all $a \neq b \in A_{\leq \gamma}$;
  \item \label{lem:properties-A:cover}
      Consider $c \in \FS(X)_{\leq \gamma} - \FS(X)^\IR$ such
      that $c(t) \in K$ for some $t\in \IR$, 
      where $K$ is the compact subset
      from Convention~\ref{con:(X,G,K)}.
      Then there are $a \in A_{\leq \gamma}$ and $y \in Y_a$
      such that $\Phi_\IR(c) = q_a^{-1}(y)$.
  \end{enumerate}
\end{lemma}

\begin{proof}~\ref{lem:properties-A:cofinite} 
  We proceed by contradiction and assume that there
  are $a_1,a_2,\ldots$ in $A_{\leq \gamma}$ such that
  $Ga_i \cap Ga_j = \emptyset$ if $i \neq j$.
  Then there are $g_i \in a_i$ and $c_i \in \FS(X)_{a_i}$ such
  that $c_i$ is an axis for $g_i$.
  After replacing $a_i$ by $h_ia_i$, $g_i$ by $h_ig_ih_i^{-1}$ and 
  $c_i$ by $h_i c_i$ for suitable $h_i \in G$ we can assume that
  $c_i(0) \in K_\gamma$, where $K_\gamma$ is the compact subset of $X$
  from Lemma~\ref{lem:K_gamma}.
  Lemma~\ref{lem:evaluation-is-proper} implies that
  $\hat K_\gamma := \{ c \in \FS(X) \mid c(0) \in K_\gamma \}$
  is also compact.
  Thus we can pass to a subsequence and arrange that 
  $c_i \to c \in \FS(X)$.
  Lemma~\ref{lem:seq-in-FS} yields a contradiction.
  \\[1ex]~\ref{lem:properties-A:GFS_a-closed}
  By Lemma~\ref{lem:seq-in-FS} $G \cdot \FS_a \subseteq \FS(X)$ is
  closed.
  Thus $K_a := \hat K_\gamma \cap G \cdot \FS_a$ is compact. 
  Now we get $G \cdot \FS_a = G \cdot \hat K_\gamma \cap G \cdot \FS_a
       = G \cdot (\hat K_\gamma \cap G \cdot \FS_a)
       = G \cdot K_a$ using Lemma~\ref{lem:K_gamma}.
  \\[1ex]~\ref{lem:properties-A:a-neq-b}
  We proceed by contradiction and assume that for
  every $n$, there are $a_n \neq b_n \in A_{\leq \gamma}$
  and $c_n \in \FS_{a_n}$, $d_n \in \FS_{b_n}$ such that
  $d_\FS(c_n,d_n) < 1/n$.
  Because of~\ref{lem:properties-A:GFS_a-closed}
  there are $h_n \in G$ such that a subsequence of 
  $h_n c_n$ converges to some $c \in \FS(X)$.
  After replacing $a_n$ by $h_n a_n$, $b_n$ by $h_n b_n$,
  $c_n$ by $h_n c_n$, $d_n$ by $h_n d_n$ and after passing
  to a suitable subsequence we can thus assume that
  $c_n \to c \in \FS(X)$. 
  Then also $d_n \to c$.
  Using Lemma~\ref{lem:seq-in-FS} for the sequence $(c_n)_{n \in \IN}$
  we can after passing to a subsequence assume that $a_n = a$ is constant
  and that $c \in \FS_a$.
  Using Lemma~\ref{lem:seq-in-FS} for the sequence $(d_n)_{n \in \IN}$
  we can after passing to a further subsequence assume that
  $b_n = b$ is constant and $c \in \FS_b$.
  Thus $\FS_b \cap \FS_a \neq \emptyset$.
  From the definition of $A_{\leq \gamma}$ it is immediate
  that this implies $a = b$, contradicting $a_n \neq b_n$.  
  \\[1ex]~\ref{lem:properties-A:cover}
  Let $c \in \FS(X)_{\leq \gamma} - \FS(X)^\IR$.
  Then we can find for each natural number $n$ a real number $\tau_n$ and $g_n \in G$ 
  satisfying $\gamma + \frac{1}{n} > \tau_n > 0$ and $g_n c(t) = c(t + \tau_n)$.
  In particular, $c$ is a geodesic that is an axis for each $g_n$
  and $l({g_n}) = \tau_n$.
  After passing to a subsequence we can assume that $\tau_n \to \tau_0$.
  Thus, $d_X(g_n c(0), c(\tau_0) ) \to 0$.
  Because of Lemma~\ref{lem:proper-action-constant-g_n}
  $l(g_n) = \tau_n = \tau_0$ for infinitely many $n$.
  For such an $n$ we have 
  $g_n \in G^{\hyp}_{\leq \gamma}$. Now assume additionally that
  $c(t) \in K$ for some $t$.
  If $a$ is the equivalence class of such a $g_n$,
  then $c \in \FS_a$ and $\Phi_\IR(c) = q_a^{-1}(q_a(c))$.
\end{proof}

\begin{proposition}
  \label{prop:Y_a}
  Let $\gamma > 0$ and $a \in A_{\leq \gamma}$.
  Then
  \begin{enumerate}
  \item \label{prop:Y_a:closed_image}
      $p_a \colon \FS_a \to X$ is an isometric embedding
      with closed image;
  \item \label{prop:Y_a:metric}
      there is a $G_a$-invariant metric $d_a$ on $Y_a$ that
      generates the topology; with this metric
      $Y_a$ is a proper metric space;
  \item \label{prop:Y_a:product}
      there is $\tau_a \colon \FS_a \to \IR$, such that
      $c \mapsto (q_a(c),\tau_a(c))$ defines an
      isometry $\FS_a \to Y_a \x \IR$ 
      which is compatible with the flow, i.e.,
      $\tau_a(\Phi_t(c)) = \tau_a(c) + t$
      for $t \in \IR$, $c \in \FS_a$;
  \item \label{prop:Y_a:VCyc}
      for $y \in Y_a$, $G_y := \{ g \in G_a \mid gy=y \}$
      is virtually cyclic of type I and $G_a y \subseteq Y_a$
      is discrete;
  \end{enumerate}
\end{proposition}

\begin{proof}~\ref{prop:Y_a:closed_image}
  If $c$ and $d$ are parallel, then $d_X(c(t),d(t))$ is constant 
  by definition. 
  An easy computation shows $d_X(c(0),d(0)) = d_\FS(c,d)$.
  Thus $p_a$ is an isometry.
  It remains to show that $p_a(\FS_a)$ is closed.
  Let $c_n \in \FS_a$ such that $c_n(0) \to x \in X$.
  Because $c \mapsto c(0)$ is proper 
  (Lemma~\ref{lem:evaluation-is-proper})
  we can pass to a subsequence and assume that $c_n \to c$
  in $\FS(X)$.
  Then $c \in \FS_a$ by Lemma~\ref{lem:seq-in-FS}
  and $c(0) = x$.  
  \\[1ex]\ref{prop:Y_a:metric} and~\ref{prop:Y_a:product} 
  Let $\FS_a^+$ be the subset of all $d$ that are parallel to
  some (and therefore all) $c \in \FS_a$.
  Define 
  $$p_a^+ \colon \FS_a^+ \to X, \quad c \mapsto c(0).$$
  By the argument from the proof of assertion~\ref{prop:Y_a:closed_image}
  $p_a^+ \colon \FS_a^+ \to X$ is an isometric embedding.  
  It follows from~\cite[II.2.14 on p.183]{Bridson-Haefliger(1999)}
  that there is convex subspace $Y^+_a$ of $X$
  (which is therefore a $\CAT(0)$-space)
  and an isometry $\psi \colon Y_a^+ \x \IR \to \FS_a^+$
  such that $\Phi_t(\psi(y,s)) = \psi(y,s+t)$ for
  $s,t \in \IR$ and $y \in Y_a^+$.    
  This identifies $Y_a$ with a subspace of $Y_a^+$
  and provides the metric on $Y_a$.
  If $g \in G_a$, then the action of $g$ on $X$ permutes 
  the images of geodesics $c \in \FS_a^+$.
  It follows from~\cite[I.5.3(4) on p.56]{Bridson-Haefliger(1999)}
  that the induced action of $g$ on $Y_a^+$ 
  (and therefore also the induced action on $Y_a$)
  is isometric. Since $X$ is proper by assumption,
  $\FS_a$ is proper by assertion~\ref{prop:Y_a:closed_image}.
  Since $\FS_a$ is isometric to $Y_a \times \IR$, the metric space
  $Y_a$ is proper.
  \\[1ex]~\ref{prop:Y_a:VCyc}
  By Lemma~\ref{lem:type-I-VCyc}
  $G_y$ is virtually cyclic of type I.
  We proceed by contradiction to show that $G_a y$ is
  discrete. Assume that there are $g_n \in G_a$, $n \in \IN$
  such that $g_n y \neq g_m y$ if $n \neq m$ and 
  $g_n y \to y_0$. 
  Pick $c \in \FS_a$ such that $q_a(c) = y$ and $\tau_a(c) = 0$.
  Pick $g \in a \subseteq G_{\leq \gamma}^{\hyp}$ such that
  $c$ is an axis for $g$.
  Then $g c = \Phi_{l(g)}(c)$,
  see Lemma~\ref{lem:basic_prop_l_plus_Min}~\ref{lem:basic_prop_l_plus_Min:hyperbolic}.
  Thus $\tau_a(g_n g g_n^{-1} (g_n c)) =  l(g) + \tau_a( g_nc )$
  and $(g_ngg_n^{-1})g_n y = g_n y$.
  By replacing  $g_n$ by $(g_n g g_n^{-1})^{l_n} g_n$ for some $l_n \in \IZ$
  we can arrange that $\tau_a (g_n c ) \in [0,l(g)]$. By passing to a subsequence
  we can arrange that the sequence $\tau_a (g_n c )$ converges.
  Since $q_a(g_n c) = g_n c$ converges to $y_0$, we conclude 
  from~\ref{prop:Y_a:product} that $g_n c \to d$ as 
  $n \to \infty$ for some $d \in \FS(X)$.
  Now Lemma~\ref{lem:proper-action-constant-g_n} implies
  that $g_n c = g_m c$ for infinitely many $n,m$.
  This contradicts $g_n y \neq g_m y$ for $n \neq m$.
\end{proof}

In the proofs of the next two results we will denote
by $\pi_a \colon Y_a \to G_a\backslash Y_a$ the
quotient map.
We point out that $\pi_a$ is open, since for any open 
subset $V \subseteq Y_a$ the subset 
$\pi_a^{-1}(\pi_a(V)) = \bigcup_{g \in G_a} g \cdot V$ is open. 

\begin{lemma} \label{lem:dim(Ga_backslash_Y_a_le_dim(X)}
  We have
  $$\dim(G_a \backslash Y_a) \le \dim(X).$$
\end{lemma}

\begin{proof}
  Let $K_a \subseteq G \cdot \FS_a$ be compact such that
  $G \cdot K_a = G \cdot \FS_a$, see 
  Lemma~\ref{lem:properties-A}~\ref{lem:properties-A:GFS_a-closed}.
  Using Lemma~\ref{lem:properties-A}~\ref{lem:properties-A:a-neq-b}  
  we conclude that there is a compact subset $K'_a \subseteq \FS_a$
  such that $\FS_a = G_a \cdot K'_a$. 
  Let $K_a'' \subseteq Y_a$  be the compact subset $q_a(K_a')$. 
  Define $U := B_1(K_a'')$.
  Then $U$ is an open subset of $Y_a$ with $Y_a = G_a \cdot U$.
  Let $i \colon U \to Y_a$ be the inclusion. 
  Since $U$ is open, $i$ is open.
  As pointed out above, $\pi_a \colon Y_a\to G_a\backslash Y_a$ is open.
  Hence the composite
  $\pi_a \circ i \colon U \to G_a \backslash Y_a$ is open and surjective. 
  Since $Y_a$ is a proper metric space by 
  Proposition~\ref{prop:Y_a}~\ref{prop:Y_a:metric}
  and $K_a'' \subseteq Y_a$ is compact, 
  the set $\overline{B}_1(K_a'')$   is compact subset of $Y_a$. 
  Since for every $y \in Y$ the
  orbit $G_ay \subseteq Y_a$ is discrete by 
  Proposition~\ref{prop:Y_a}~\ref{prop:Y_a:VCyc},
  the intersection $\overline{B}_1(K_a') \cap G_ay$ 
  and hence also the intersection
  $U \cap G_ay$ is finite. 
  Hence the composite
  $\pi_a \circ i \colon U \to G_a \backslash Y_a$ is finite-to-one. 
  Its source is a metric as $U$ is a subspace of the metric space $Y_a$. 
  Since $G_a$ acts  isometrically on $Y_a$ and for every $y \in Y$ the
  orbit $G_ay \subseteq Y_a$ is discrete by 
  Proposition~\ref{prop:Y_a}~\ref{prop:Y_a:VCyc}, the 
  quotient $G_a \backslash Y_a$ is also a metric space. 
  Since every metric space is paracompact by Stone's Theorem
  (see~\cite[Theorem~4.3 in Chapter~VI on p.~256]{Munkres(1975)}, 
  the composite
  $\pi_a \circ i \colon U \to G_a \backslash Y_a$ is a finite-to-one 
  open surjective map of paracompact spaces. 
  Hence we conclude 
  from~\cite[4.1 on p.35]{Nagami(1960mappings-of-finite)} 
  $$\dim(U) = \dim(G_a\backslash Y_a).$$
  
  Since $Y_a$ is a proper metric space by 
  Lemma~\ref{prop:Y_a}~\ref{prop:Y_a:metric}
  it is locally compact and can be written as the countable 
  union of compact subspaces
  and hence contains a countable dense subset. Hence the open subset
  $U$ is locally compact and has a countable basis for its topology. 
  Since any compact subset $L \subset U$ is a 
  closed subset of $X$ and satisfies
  $\dim(L) \le \dim(X)$, we conclude
  $\dim(U) \le \dim(X)$ 
  from~\cite[Exercise~9 in Chapter~7.9 on p.~315]{Munkres(1975)}.
  This finishes the proof of Lemma~\ref{lem:dim(Ga_backslash_Y_a_le_dim(X)}.
\end{proof}

\begin{proposition}
  \label{prop:cover-Y_a}
  Let $\gamma > 0$ and $a \in A_{\leq \gamma}$.
  There is an open $\VCyc$-cover $\calv_a$ of $Y_a$ such that
  \begin{enumerate}
  \item \label{prop:cover-Y_a:dim} 
      $\dim \calv_a \leq \dim X$;
  \item \label{prop:cover-Y_a:invariant} 
      $\calv_a$ is $G_a$-invariant, i.e.,
      $g \cdot V \in \calv_a$ if $g \in G_a$ and $V \in \calv_a$;
  \item \label{prop:cover-Y_a:finite} 
      $G_a \backslash \calv_a$ is finite.
  \end{enumerate}
\end{proposition}

\begin{proof}
  Because of Proposition~\ref{prop:Y_a}~\ref{prop:Y_a:metric}
  and~\ref{prop:Y_a:VCyc} for any $y \in Y_a$,
  the open ball of sufficiently small radius is  
  $\VCyc$-neighborhood for $y$.
  Pick for each $y$ such a ball $V_y$.
  Because $\pi_a \colon Y_a \to G_a\backslash Y_a$
  is open,
  $\{ \pi_a(V_y) \mid y \in Y_a \}$ is an open cover of 
  $G_a \backslash Y_a$. 
  By Lemma~\ref{lem:dim(Ga_backslash_Y_a_le_dim(X)}
  it has a refinement $\calw$
  whose dimension is bounded by $\dim X$.
  The $G_a$-action on $Y_a$ is cocompact because it
  is cocompact on $\FS_a$, see 
  Lemma~\ref{lem:properties-A}~\ref{lem:properties-A:GFS_a-closed}.
  Therefore $G_a \backslash Y_a$ is compact.
  Thus we may assume that $\calw$ is finite.
  For any $W \in \calw$ pick $y_W \in \FS_a$ such that 
  $W \subseteq \pi_a(V_{y_W})$.
  Now define
  $\calv_a := \{ \pi_a^{-1}(W) \cap g V_{y_W} 
               \mid W \in \calw, g \in G_a \}$. 
  This is an open $\VCyc$-cover because each $V_y$ is an open 
  $\VCyc$-set.
  Its dimension is bounded by $\dim X$, because the dimension
  of $\calw$ is bounded by $\dim X$ and because for $g \in G_a$
  and $y \in Y$ we have either $V_y =  g V_y$ or
  $V_y \cap g V_y = \emptyset$.
  It is $G_a$-invariant because each $\pi_a^{-1}(W)$ is $G_a$-invariant.
  Finally, $G_a \backslash \calv_a$ is finite because $\calw$ is finite.
\end{proof}


\subsection{The cover $\calv$}

\begin{lemma}
  \label{lem:extending-opens}
  Let $(Z,d_Z)$ be a metric space
  with an action of a group $H$ by isometries.
  Let $A$ be a $H$-invariant subspace.
  For $\emptyset \neq U \subsetneq A$, we define
  \[
    Z(U) := \{ z \in Z \; | \; d_Z(z, U) < d_Z(z, A-U) \} 
  \]
  and set $Z(A) := Z$, $Z(\emptyset) := \emptyset$.
  Then for $U,V \subseteq A$,
  \begin{enumerate}
  \item \label{lem:extending-opens:open}
      $Z(U)$ is open in $Z$;
  \item \label{lem:extending-opens:cap}
      $Z(U \cap V)  =  Z(U) \cap Z(V)$;
  \item \label{lem:extending-opens:intersection-A}
      $Z(U) \cap A = U$ holds if and only if $U$ is open in $A$;
  \item \label{lem:extending-opens:g}
      for all $g \in H$ we have $Z(gU)  =  gZ(U)$.
  \end{enumerate}
\end{lemma}

\begin{proof}~\ref{lem:extending-opens:open}
  Either $Z(U)$ is $\emptyset$ or $Z$ or  
  can be written as the preimage of $(0,\infty)$
  for a continuous function on $Z$. 
  Hence $Z(U)$ is open for every $U \subseteq Z$.
  \\[1ex]~\ref{lem:extending-opens:cap}
  This is obviously true if $U = A$, $U = \emptyset$,
  $V = \emptyset$ or $V = A$ holds, so we can assume without
  loss of generality $\emptyset \not= U \not= A$ and
  $\emptyset \not= V \not= A$ in the sequel.
  Recall that $d_Z(z,U) := \inf\{d_Z(z,u) \mid u \in U\}$ for $z \in Z$
  and $\emptyset \not= U \subseteq Z$.
  One easily checks that $d_Z(z,U), d_Z(z,V) \le d_Z(z,U\cap V)$ and
  $d_Z(z,A-(U \cap V)) = \min\{d_Z(z,A-U),d_Z(z,A-V)\}$ hold
  for $z \in Z$ and open subsets $U,V \subseteq A$.
  This implies $Z(U \cap V) \subseteq Z(U) \cap Z(V)$.

  It remains to show $Z(U) \cap Z(V)  \subseteq Z(U \cap V)$.
  Let $z \in Z(U) \cap Z(V)$.
  Because of $d_Z(z,U) < d_Z(z,A-U)$ there is $u \in U$ with
  $d_Z(z,u) < d_Z(z,A-U)$.
  Because of $d_Z(z,V) < d_Z(z,A-V)$ there is $v \in V$ with
  $d_Z(z,v) < d_Z(z,A-V)$.
  If $u \not\in V$ then $d_Z(z,v) < d_Z(z,u)$.
  If $v \not\in U$ then $d_Z(z,u) < d_Z(z,v)$. In particular
  we have $u \in V$ or $v \in U$.

  Suppose that $u \not\in V$. Then $v \in U$ and
  $d_Z(z,v) < d_Z(z,u) < d_Z(z,A - U)$.
  Thus
  \begin{multline*}
   d_Z(z, V \cap U) \leq d_Z(z,v) 
        < \min\{d_Z(z,A-U),d_Z(z,A-V)\} \\
                            = d_Z(z,A-(U \cap V))
  \end{multline*}
  and $z \in Z(U \cap V)$. Analogously one shows
  $z \in Z(U \cap V)$ if $v \notin U$.
  It remains to treat the case, where $u \in V$ and $v \in V$, 
  or, equivalently, where 
  $u,v \in U \cap V$. 
  We may assume without loss of
  generality $d_Z(z,v) \leq d_Z(z,u)$.
  Then $d_Z(z,v) \leq d_Z(z,u) < d_Z(z,A - U)$.
  Thus
  \begin{multline*}
  d_Z(z, V \cap U) \leq d_Z(z,v) < \min\{d_Z(z,A-U),d_Z(z,A-V)\}
                         \\  = d_Z(z,A-(U \cap V))
  \end{multline*}
  and $z \in Z(U \cap V)$. 
  This proves $Z(U \cap V) = Z(U) \cap Z(V)$.
  \\[1ex]~\ref{lem:extending-opens:intersection-A}
  If $ U = \emptyset$ or $U = A$, this is obvious so that
  it suffices to treat the case $\emptyset \not= U \not= A$.
  If $Z(U) \cap A = U$, then $U$ is open in $A$, 
  because $Z(U)$ is open in $Z$ by~\ref{lem:extending-opens:open}. 
  
  Assume that $U$ is open in $A$. 
  Consider  $z \in U$. Since $U \subseteq A$ is open, there exists
  $\epsilon > 0$ such that 
  $\{y \in Z \mid d_Z(z,y) < \epsilon\} \cap A \subseteq U$.
  Hence $d_Z(x,z) \ge \epsilon$ for $x \in A-U$ and hence 
  $d_Z(z,A-U) \ge \epsilon > 0 = d_Z(z,U)$.
  Therefore $z \in Z(U) \cap A$. 
  Consider $z \in Z(U) \cap A$. 
  Since $d_Z(z,U) < d_Z(z,A-U)$ implies
  $d_Z(z,A-U) > 0$ and hence $z \not\in A-U$, we conclude $z \in U$.
  This proves $Z(U) \cap A = U$.
  \\[1ex]~\ref{lem:extending-opens:g} 
  This is obvious as $H$ acts by isometries.
\end{proof}

\begin{lemma}
  \label{lem:cover-constant-geodesics}
  There is $\e_\IR > 0$ and a $G$-invariant cofinite 
  collection $\calv_\IR$ of open $\Fin$-subsets of $\FS(X)$ such that
  \begin{enumerate}
  \item $\dim \calv_\IR < \infty$;
  \item if the image of $c \in \FS(X)^\IR$ as a generalized geodesic in $X$ 
      (which is a point) intersects (and is therefore contained in) 
      $G \cdot K$, then there is $U \in \calv_\IR$ such that 
      $B_{\e_\IR}(c) = B_{\e_\IR} (\Phi_\IR(c)) \subseteq U$.  
  \end{enumerate}
\end{lemma}

\begin{proof}
  By Lemma~\ref{lem:evaluation-is-proper} the evaluation map $\FS(X) \to X$ 
  defined by $c \mapsto c(0)$ is proper.
  As $\FS(X)^\IR$ is closed, this implies that
  there is a compact subset $K_\IR \subseteq \FS(X)^\IR$ such that  
  $G \cdot K_{\IR}$ contains
  all $c \in \FS(X)^\IR$ whose image lies in $G \cdot K$.  
  For any $c \in K_\IR$ there is an open $\Fin$-subset $V_c$ of 
  $\FS(X)$ containing $c$,
  because the action of $G$ on $\FS(X)$ is proper 
  (see~Proposition~\ref{prop:cocompact}).
  Because $K_\IR$ is compact there is $\e_\IR > 0$ and a finite subset 
  $\Lambda$ of $K_\IR$ such that for any $c \in K_\IR$ there is 
  $\lambda_c \in \Lambda$ such that $B_{\e_\IR}(c) \subseteq V_{\lambda_c}$.
  Now set $\calv_\IR := \{ g V_\lambda \mid g \in G, \lambda \in \Lambda \}$. 
  This is a finite dimensional collection, because the $g V_\lambda$ are 
  $\Fin$-sets and $\Lambda$ is finite.    
\end{proof}

\begin{proof}[Proof of Theorem~\ref{the:covering_of_FS(X)_gamma}]
  Let $R \subseteq A_{\leq \gamma}$ be a subset that contains exactly one
  element from each orbit of the $G$-action. Then $R$ is finite
  by Lemma~\ref{lem:properties-A}~\ref{lem:properties-A:cofinite}.
  For each $a \in R$, let $\calv_a$ be a covering of $Y_a$   
  satisfying the assertions from Proposition~\ref{prop:cover-Y_a}.
  Let $\calw_a := (q_a)^{-1} \calv_a = 
          \{ q_a^{-1}(V) \mid V \in \calv_a \}$.
  For $b \in A_{\leq \gamma}$ pick $a \in R$ and $g \in G$,
  such that $ga = b$ and set 
  $\calw_b := g(\calw_a) = \{ gW \mid W \in \calw_a \}$.
  (This does not depend on the choice of $g$, as $q_a$ is
  $G_a$-equivariant and $\calv_a$ is $G_a$-invariant.)
  By Lemma~\ref{lem:properties-A}~\ref{lem:properties-A:a-neq-b} 
  there is $\delta > 0$ such that, if we set $U_a := B_\delta(\FS_a)$, then
  $U_a \cap U_b = \emptyset$ for $a \neq b \in A_{\leq \gamma}$.
  We now use Lemma~\ref{lem:extending-opens} to extend
  the $W \in \calw_a$ to open subsets of $\FS(X)$
  and define the collection $\calu$ by
  \begin{equation*}
    \calu := \bigcup_{a \in A_{\leq \gamma}} 
     \{ Z_a(W) \cap U_a \mid W \in \calw_a \},
  \end{equation*}
  where $Z_a(W) := \{ c \in \FS(X) \; | \; 
                    d_\FS(c, \FS_a) < d_\FS(c, \FS(X)-\FS_a) \}$.
  Define the desired collection of open subsets of $\FS(X)$ by
  \[
  \calv := \calu \cup \calv_{\IR},
  \]
  where $\calv_\IR$ is from Lemma~\ref{lem:cover-constant-geodesics}.
  It remains to show, that $\calv$ has the desired properties.
  \\[1ex]~\ref{the:covering_of_FS(X)_gamma:VCyc}
  The members of each $\calv_a$ are open $\VCyc$-sets
  with respect to the $G_a$-action by Proposition~\ref{prop:cover-Y_a}.
  The $q_a$ are continuous and $G_a$-equivariant by construction.
  Thus the members of each $\calw_a$ are also open $\VCyc$-sets
  with respect to the $G_a$-action. 
  For each $a$, $\FS_a$ is a $G_a$-set.
  Because the $U_a$ are mutually disjoint, 
  each $U_a$ is an $G_a$-set as well.
  By Lemma~\ref{lem:extending-opens}  each
  $Z_a(W) \cap U_a$ with $W \in \calw_a$
  is an open $\VCyc$-subset of $\FS(X)$ with respect to the $G$-action.
  \\[1ex]~\ref{the:covering_of_FS(X)_gamma:invariance}
  Each $\calv_a$ is $G_a$-invariant.
  The union of the $\calw_a$ is $G$-invariant, because the $q_a$
  are $G_a$-equivariant.
  Thus by Lemma~\ref{lem:extending-opens} the collection of all
  $Z_a(W)$ is $G$-invariant.
  The collection of the $U_a$ is $G$-invariant.
  Therefore $\calu$ is $G$-invariant.
  Since $\calu_\IR$ is $G$-invariant this implies
  that $\calv$ is $G$-invariant.
  \\[1ex]~\ref{the:covering_of_FS(X)_gamma:finite}
  Each $G_a \backslash \calv_a$ is finite by
  Proposition~\ref{prop:cover-Y_a}~\ref{prop:cover-Y_a:finite}.
  Therefore $G \backslash \calu$ is finite.
  Since $G \backslash \calu_\IR$ is finite, $G \backslash \calv$
  is finite.
  \\[1ex]~\ref{the:covering_of_FS(X)_gamma:dim} 
  The $U_a$ are mutually disjoint. 
  For each $a$, $\dim \calw_a = \dim \calv_a$.
  Using Proposition~\ref{prop:cover-Y_a}~\ref{prop:cover-Y_a:dim}
  and Lemma~\ref{lem:extending-opens} we get therefore 
  \begin{eqnarray*}
    \dim \calu 
    &  = & 
    \max_{a \in A} \dim \{ Z_a(W) \cap U_a \mid W \in \calw_a \}
    \\
    & = & 
    \max_{a \in R} \dim \calv_a \leq \dim X. 
  \end{eqnarray*}
  Put $M : = 1 + \dim(\calv_{\IR}) + \dim(X)$.
  This number is independent of $\gamma$ and
  \[
  \dim(\calv) \le 1  + \dim(\calv_{\IR}) + \dim(\calu) = M.
  \]
  \\[1ex]~\ref{the:covering_of_FS(X)_gamma:flow}
  Let $c \in \FS(X)$ such that $c$ intersects $G \cdot K$. 
  If $c \in \FS(X)^\IR$, then 
  $B_{\e_\IR}(\Phi_{\IR}(c)) = B_{\e_\IR}(c)  \subseteq U$
  for some $U \in \calv_\IR$, where $\e_\IR$ is from 
  Lemma~\ref{lem:cover-constant-geodesics}. Hence it 
  remains to show that there is $\e_\calu$ such that for any 
  $c \in \FS(X)_{\leq \gamma}- \FS(X)^\IR$ such that $c$ 
  intersects $G \cdot K$, there is $U \in \calu$ satisfying
  $B_{\e_\calu}(\Phi_{[-\gamma,\gamma]}(c)) \subseteq U$ because then we can take
  $\epsilon = \min\{\epsilon_{\IR}, \epsilon_{\calu}\}$.
  
  Now suppose that the desired $\epsilon_{\calu}$ does not exists, 
  i.e., there are $c_n \in \FS(X)_{\leq \gamma} - \FS(X)^\IR$, 
  $d_n \in \FS(X)$ and $t_n \in \IR$ for $n \ge 1$ such that 
  $c_n$ intersects $ G \cdot K$,
  $d_\FS(\Phi_{t_n}(c_n),d_n) < 1/n$ and $d_n \not\in U$ for all 
  $U \in \calu$ that contain $\Phi_{[-\gamma,\gamma]}(c_n)$.
  Choose $a_n \in R$ with $c_n \in G \cdot \FS_{a_n}$.
  As $R$ is finite, we arrange by passing to a subsequence
  that there is $a \in R$ with $a_n = a$ for all $n$.
  
  Because of Lemma~\ref{lem:properties-A}~\ref{lem:properties-A:GFS_a-closed}
  there are $g_n \in G_a$ such that all $g_n c_n$ 
  are contained in a compact set
  $K_a \subset G \cdot \FS_a$.
  After passing to a subsequence and replacing $c_n$ by $g_n c_n$ and 
  $d_n$ by $g_n d_n$ we can assume that $c_n \to c$ for 
  some $c \in K_a \subset G \cdot \FS_a$. 
  Choose $g \in G$ with $g^{-1} c \in \FS_{a}$.
  Choose $V \in \calv_a$ with $q_a(g^{-1}c) \in V$. 
  Then $g^{-1}c \in q_a^{-1}(V)$ 
  and hence $\Phi_{\IR}(g^{-1}c) \subset q_a^{-1}(V)$.
  Thus $\Phi_\IR(c) \subseteq U$ for some $U \in \calu$, namely for
  $U = g \cdot Z_{a}(W) \cap U_{ga}$, where we set $W := q_a^{-1}(V)$. 
  
  Passing to a further subsequence we can arrange
  that $t_n \to t_0$ for  some $t_0 \in [-\gamma,\gamma]$.
  Then also $d_n \to \Phi_{t_0}(c)$. Hence there is $n_0$ such that
  $d_n \in U$ for $n \ge n_0$,
  because $d_n \to \Phi_{t_0}(c) \in U$. 

  The set $\Phi_{[-\gamma,\gamma]}(c)$
  is a compact subset of the open set $U$. Hence we can find $\delta > 0$ such that
  $B_{\delta}\bigl(\Phi_{[-\gamma,\gamma]}(c)\bigr) \subseteq U$. 
  We have $d_X\bigl(\Phi_t(c),\Phi_t(c_n)\bigr) \le e^{\tau} \cdot d_{FS(X)}(c_n,c)$ 
  for $n \ge 1$ and $t \in [-\tau,\tau]$ by 
  Lemma~\ref{lem:Phi_well-defined}. Since there exists $n_1$ such that
  $ d_{FS(X)}(c_n,c) < e^{-\tau} \cdot \delta$ for $n \ge n_1$, we get
  $\Phi_{[-\gamma,\gamma]}(c_n) \subseteq U$ for $n \ge n_1$, a contradiction.
  This finishes the proof of Theorem~\ref{the:covering_of_FS(X)_gamma}.
\end{proof}


\typeout{---------- Flow spaces and S-long covers ----------------}

\section{Flow spaces and $S$-long covers}
\label{sec:flow-spaces-and-S-long-covers}

\begin{summary*}
  In Definitions~\ref{def:at-infinity_plus_periods} 
  and~\ref{def:contracting-transfers} 
  we formulate two conditions for
  a flow space $\FS$ with an action of a group $G$.
  The first condition asks for the existence of 
  long covers of uniformly bounded dimension 
  for a subset of $\FS$ that contains the periodic
  orbits of the flow and is large in the sense that 
  its complement is cocompact for the action of $G$.
  (In our application later the action of $G$ on $\FS$
  will be cocompact and we will get the second part of this condition 
  for free; we expect however that this condition can also
  be verified in situations where the action is not cocompact.)  
  The second condition concerns the dynamic of the flow
  with respect to a suitable homotopy action. 
  Our main result is Proposition~\ref{prop:wide-covers-with-flow-spaces}
  which asserts that $G$ is transfer reducible, provided 
  the two conditions are satisfied. 
\end{summary*}

In this section we fix the following convention, 
compare~\cite[Convention~1.3]{Bartels-Lueck-Reich(2008cover)}.

\begin{convention}
  \label{conv:FS}
  Let
  \begin{itemize}
    \item $G$ be a group;
    \item $\calf$ be a family of subgroups of $G$;
    \item $(\FS,d_{\FS})$ be a locally compact metric space with a proper 
      isometric $G$-action;     
    \item $\Phi \colon \FS \times \IR \to \FS$ be a flow.
  \end{itemize}
  We assume that the following conditions are satisfied:
  \begin{itemize}
    \item $\Phi$ is $G$-equivariant;
    \item $\FS - \FS^{\IR}$  is locally connected; 
    \item $k_G :=  \sup\{|H| \mid H \subseteq G
             \;\text{subgroup with finite order } |H|\} < \infty$;
    \item $\dim ( \FS-\FS^{\IR} ) < \infty$;
    \item the flow is uniformly continuous in the following sense:
      for $\alpha > 0$ and $\e > 0$ there is $\delta > 0$ such that 
      \begin{equation}
        \label{eq:flow-uniform-continuous}
          d_\FS(z,z') \leq \delta, \tau \in [-\alpha,\alpha]
          \implies
          d_\FS(\Phi_{\tau}(z),\Phi_{\tau}(z')) \leq \e.
      \end{equation}
  \end{itemize}
\end{convention}

For a subset $I \subseteq \IR$ we set 
$\Phi_I(z) := \{\Phi_t(z) \mid t \in I\}$. 
For $z \in \FS$ we define its \emph{$G$-period}
\begin{equation*}
 \peri{G}{\Phi}(z) =  
    \inf\{ \tau  \mid \tau > 0, \; \exists \; g \in G \;\text{with }
              \Phi_{\tau}(z) = gz\} \quad \in [0,\infty],
\end{equation*}
where the infimum over the empty set is defined to be $\infty$. 
Obviously $\peri{G}{\Phi}(z) = 0$ if and only if $z \in \FS^{\IR}$. 
If $L \subseteq \FS$ is an orbit of the flow $\Phi$, define its
\emph{$G$-period} by
\begin{equation*}
  \peri{G}{\Phi}(L)  :=  \peri{G}{\Phi}(z) 
\end{equation*}
for any choice of $z \in \FS$ with $L = \Phi_{\IR}(z)$.
For $\gamma \geq 0$ put
\begin{eqnarray}
  \FS_{>\gamma} & := & \{z \in \FS \mid \peri{G}{\Phi}(z) > \gamma \};
  \label{FS_greater_gamma}
  \\
  \FS_{\le \gamma} & := & \{z \in \FS \mid \peri{G}{\Phi}(z) \le  \gamma \}.
  \label{FS_less_or_equal_gamma}
\end{eqnarray}

\begin{definition}
   \label{def:at-infinity_plus_periods}
    We will say that $\FS$ 
    \emph{admits long $\calf$-covers at infinity and periodic flow lines}
    if the following holds:
    
    There is $M > 0$ such that for every $\gamma > 0$ there is a 
    collection $\calv$ of of open $\calf$-subsets of $\FS$ and $\e > 0$  
    satisfying:
    \begin{enumerate}
         \item \label{def:at-infty_plus_periods:G-inv}
             $\calv$ is $G$-invariant: $g \in G$, 
             $V \in \calv \implies gV \in \calv$;
         \item \label{def:at-infty_plus_periods:dim}
             $\dim \calv \leq M$;
         \item \label{def:at-infty_plus_periods:covers}
           there is a compact subset $K \subseteq \FS$ such that
           \begin{itemize}
             \item $\FS_{\leq \gamma} \cap G \cdot K = \emptyset$;
             \item for $z \in \FS - G \cdot K$ there is 
                $V \in \calv$ such that 
                $B_\e(\Phi_{[-\gamma,\gamma]}(z)) \subset V$.
           \end{itemize}
    \end{enumerate}
\end{definition}

\begin{theorem}
  \label{thm:cover-compact-in-FS}
  There is $M \in \IN$ such that the following holds:
  
  For any $\alpha > 0$ there is $\gamma > 0$ such that for 
  any compact subset $K$ of $\FS_{> \gamma}$  
  there is a collection of open 
  $\VCyc$-subsets of $\FS$ such that 
  \begin{enumerate}
    \item \label{thm:cover-compact-in-FS:G-inv}
        $\calv$ is $G$-invariant: $g \in G$, 
        $V \in \calv \implies gV \in \calv$;
    \item \label{thm:cover-compact-in-FS:dim}
        $\dim \calv \leq M$;
    \item \label{thm:cover-compact-in-FS:cofinite}
        $G \backslash \calv$ is finite; 
       
    \item \label{thm:cover-compact-in-FS:covers}
        for every $z \in G \cdot K$ there is 
       $V \in \calv$ such that $\Phi_{[-\alpha,\alpha]}(z) \subset V$.
  \end{enumerate}
\end{theorem}

\begin{proof}
  This follows from the techniques used and developed in 
  Sections~2-5
  of~\cite{Bartels-Lueck-Reich(2008cover)},
  but is unfortunately note state in precisely this form.
  
  The main input is Proposition~4.1 
  of~\cite{Bartels-Lueck-Reich(2008cover)}.
  In this reference it is assumed that that $G$-acts
  cocompactly on $\FS$, but this is not used in its proof,
  mainly because the statement  concerns only a cocompact part
  of the flow space.
  (In~\cite{Bartels-Lueck-Reich(2008cover)} cocompactness of the
  action is used to conclude that the flow space is locally compact;
  note that we assumed this in Convention~\ref{conv:FS}.)
  Therefore Proposition~4.1 of~\cite{Bartels-Lueck-Reich(2008cover)} 
  is valid in the present situation as well.
  Theorem~\ref{thm:cover-compact-in-FS} can be deduced from this
  using the argument given on p.1848 
  of~\cite{Bartels-Lueck-Reich(2008cover)}.   
\end{proof}

\begin{theorem}
  \label{thm:long-thin-cover}
  Assume that $\FS$ admits long $\calf$-covers at 
  infinity and periodic flow lines
  and that $\calf$ contains the family $\VCyc$ of virtually cyclic subgroups.

  Then there is $\widehat{N} \in \IN$ such that for every $\alpha > 0$ there exists
  an open $\calf$-cover $\calu$ of $\FS$ of dimension at most $\widehat{N}$ and
  $\epsilon > 0$ (depending on $\alpha$) such that the following holds:
  \setlength\labelwidth\origlabelwidth
  \begin{enumerate}
  \item 
    For every $z \in \FS$ there is
    $U \in \calu$ such that
    $B_{\epsilon}\big(\Phi_{[-\alpha,\alpha]}(z)\big) \subseteq U$.
  \item $\calu / G$ is finite.
  \end{enumerate}
\end{theorem}

\begin{proof}
  Let $M_0$ be the number $M$ appearing in 
  Definition~\ref{def:at-infinity_plus_periods}
  and $M_1$ be the number $M$ appearing in
  Theorem~\ref{thm:cover-compact-in-FS}.
  We set $\widehat{N} := M_0 + M_1 + 1$.
  Theorem~\ref{thm:cover-compact-in-FS} 
  also provides a number $\gamma$ depending on $\alpha$.
  We can assume that $\gamma \geq \alpha$.
  Since $\FS$ admits long $\calf$-covers at infinity and periodic
  flow lines we can find a collection $\calv_0$ of open 
  $\calf$-subsets of $\FS$ and $\e_0 > 0$ such 
  that~\ref{def:at-infty_plus_periods:G-inv} 
  to~\ref{def:at-infty_plus_periods:covers} from 
  Definition~\ref{def:at-infinity_plus_periods} hold.
  
  Next we apply Theorem~\ref{thm:cover-compact-in-FS} and 
  obtain a collection $\calv_1$ of open $\VCyc$-subsets of $\FS$
  and a compact subset $K \subseteq \FS$
  such that~\ref{thm:cover-compact-in-FS:G-inv} 
  to~\ref{thm:cover-compact-in-FS:covers} from
  Theorem~\ref{thm:cover-compact-in-FS} hold.
  A simple compactness argument, provided in Lemma~\ref{lem:extra-eps}
  below, shows that there is $\e_1 > 0$ such that
  for every $z \in G \cdot K$ there is $V \in \calv_1$
  such that
  \begin{equation*}
    B_{\e_1} (\Phi_{[-\alpha,\alpha]}(z)) \subset V.
  \end{equation*}
  Now set $\e := \min \{ \e_0, \e_1 \}$ and 
  $\calu := \calv_0 \cup \calv_1$.
\end{proof}

\begin{lemma}
  \label{lem:extra-eps}
  Assertion~\ref{thm:cover-compact-in-FS:covers} in 
  Theorem~\ref{thm:cover-compact-in-FS} can be strengthened to
  \begin{enumerate}
  \item[(iv')] There is $\e > 0$ such that for any $z \in G \cdot K$ 
     there is $V \in \calv$ such that 
     $B_{\e} (\Phi_{[-\alpha,\alpha]}(z)) \subset V$.
  \end{enumerate}
\end{lemma}

\begin{proof}
  Suppose that there is no such $\e$. 
  Then we can find a sequence $(z_n)_{n\ge 1}$ of points in 
  $G \cdot K$ such that 
  $B_{1/n}\big(\Phi_{[-\alpha,\alpha]}(z_n)\big) \nsubseteq U$
  holds for every $n \ge 1$ and every $U \in \calv$.   
  Because $\calv$ is $G$-invariant and $\Phi$ is $G$-equivariant, 
  we can assume without
  loss of generality that $z_n \in K$ for all $n$.
  After passing to a subsequence we can assume that 
  $z_n \to z$ as $n \to \infty$.
  Choose $V \in \calv$ with $\Phi_{[-\alpha,\alpha]}(z) \subseteq V$.  
  Since $\Phi_{[-\alpha,\alpha]}(z)$ is compact and $V$ is
  open, we can find $\mu > 0$ with
  $B_{\mu}\bigl(\Phi_{[-\alpha,\alpha]}(z)\bigr) \subseteq V$.
  Because of the uniform continuity of the flow, we can find
  $\delta > 0$ such that 
  $d_{\FS}(\Phi_{\tau}(z),\Phi_{\tau}(z')) < \mu/2$ holds for all 
  $\tau \in [-\alpha,\alpha]$ provided that
  $d_{\FS}(z,z') \le \delta$ is true. 
  Since $\lim_{n \to \infty} z_n = z$, we can find $n \ge 1$ 
  such that $d_{\FS}(z,z_n) < \delta$ and $1/n < \mu/2$ hold. 
  This implies
  $d_{\FS}(\Phi_{\tau}(z),\Phi_{\tau}(z_n)) < \mu/2$ for all $\tau
  \in [-\alpha,\alpha]$. 
  Thus
  \begin{equation*}
    B_{1/n}\big(\Phi_{[-\alpha,\alpha]}(z_n)\big)  \subseteq 
      B_{\mu}\big(\Phi_{[-\alpha,\alpha]}(z)\big) \subseteq V,
  \end{equation*}
  contradicting the assumption.
\end{proof}

\begin{definition}
  \label{def:contracting-transfers}
  We will say that $\FS$ \emph{admits contracting transfers}
  if for every finite subset $S$  of $G$ (containing $e$)  
  there exists $\beta > 0$ and $\widetilde{N} \in \IN$ 
  such that the following holds:

  For every $\delta > 0$ there is
  \begin{enumerate}
    \item \label{def:contracting-transfers:T} $T > 0$;
    \item \label{def:contracting-transfers:X} 
        a contractible compact controlled $\widetilde{N}$-dominated
        space $X$;
    \item \label{def:contracting-transfers:varphi_plus_H}
        a homotopy $S$-action $(\varphi,H)$ on $X$;
    \item \label{def:contracting-transfers:iota}
        a $G$-equivariant map $\iota \colon G \x X \to \FS$
        (where we use the left action $g \cdot (h,x) = (gh,x)$ on $G \x X$),
  \end{enumerate}
  such that
  \setlength\labelwidth\origlabelwidth
  \begin{numberlist}
    \item[\label{nl:foliated-control}] 
        for every $(g,x) \in G \x X$, $s \in S$, $f \in F_s(\varphi,H)$
        there is
        $\tau \in [-\beta,\beta]$ such that 
        \begin{equation*}
          d_\FS\left(\Phi_T(\iota(g,x)),
             \Phi_{T+\tau}(\iota(gs^{-1},f(x)))\right) \leq \delta.
        \end{equation*}
  \end{numberlist}
\end{definition}

\begin{proposition}
  \label{prop:wide-covers-with-flow-spaces}
  If $\FS$ satisfies the assumptions appearing in Convention~\ref{conv:FS},
  admits long covers at infinity and periodic orbits (see
  Definition~\ref{def:at-infinity_plus_periods}), and admits contracting transfers
  (see Definition~\ref{def:contracting-transfers}), then $G$ is transfer
  reducible over the family $\calf$ in the sense of
  Definition~\ref{def:transfer-reducible}.
\end{proposition}

The proof of Proposition~\ref{prop:wide-covers-with-flow-spaces}
will use the following the lemma.

\begin{lemma}
  \label{lem:foliated-controll-for-F_h}
  Let $\beta > 0$, $T > 0$ and $\epsilon > 0$.
  Let $S$ be a finite subset of $G$ (containing $e$).
  Set $n := |S|$ and $\alpha := 2 n \beta $. 
  Pick $\delta$ such that \eqref{eq:flow-uniform-continuous} holds. 
  Let $(\varphi,H)$ be a homotopy $S$-action on a compact metric space $X$
  and let $\iota \colon G \x X \to \FS$ be a $G$-equivariant map.
  Assume \eqref{nl:foliated-control} holds. 

  Then
  \setlength\labelwidth\origlabelwidth
  \begin{numberlist}
    \item[\label{nl:foliated-control-for-F_h}] 
       for every $(g,x) \in G \x X$ and 
       $(h,y) \in S^n_{\varphi,H}(g,x)$  
       (see Definition~\ref{def:S-action_plus_long-covers}) 
       there is
       $\tau \in [-\alpha,\alpha]$ such that 
       \begin{equation*}
         d_\FS\left( \Phi_T(\iota(g,x )),\Phi_{T+\tau}(\iota(h,y))\right)  
           \leq  2n \epsilon.
       \end{equation*}
  \end{numberlist}
\end{lemma}

\begin{proof}
  If $(h,y) \in S^n_{\varphi,H}(g,x)$,
  then there are
  $x_0,\dots,x_n \in X$,
  $a_1,b_1,\dots,a_n,b_n \in S$,
  $f_1,\widetilde{f}_1,\dots,f_n,\widetilde{f_n} \colon X \to X$,
  such that
  $x_0 = x$, $x_n = y$, $f_i \in F_{a_i}(\varphi,H)$,
  $\widetilde{f}_i \in F_{b_i}(\varphi,H)$,
  $f_i(x_{i-1}) = \widetilde{f}_i(x_i)$ and
  $h = g a_1^{-1} b_1 \dots a_n^{-1} b_n$.
  Set $g_i := g a_1^{-1} b_1 \dots a_i^{-1} b_i$.
  By \eqref{nl:foliated-control} there are 
  $\tau_1,\widetilde{\tau}_1,\dots,\tau_n,\widetilde{\tau}_n 
                                            \in [-\beta,\beta]$ 
  such that
  \begin{eqnarray*}
     d_\FS\left(\Phi_T(\iota( g_{i-1},x_{i-1} )),
        \Phi_{T+\tau_i}(\iota (g_{i-1} a^{-1}_{i},f_i(x_{i-1})))\right) 
     & \leq & \delta; \\
     d_\FS\left(\Phi_T (\iota(g_i,x_i)), \Phi_{T+\widetilde{\tau}_{i}}
                    (\iota (g_i b_i^{-1},\widetilde{f}_i(x_i)))\right)
     & \leq & \delta 
  \end{eqnarray*}
  for $i = 1,\dots, n$. 
  Put $\sigma_0 = 0$ and 
  $\sigma_i := (-\widetilde{\tau_1} + \tau_1) + \dots + 
  (-\widetilde{\tau_i} + \tau_i)$ for $i = 1,2, \dots, n$.
  Since $\sigma_i \in [-\alpha,\alpha]$, we conclude 
  for $i = 1,\dots, n$ from~\eqref{eq:flow-uniform-continuous} 
  \begin{eqnarray*}
      d_\FS\left(\Phi_{T+\sigma_{i-1}}(\iota( g_{i-1},x_{i-1} )),
           \Phi_{T+\tau_i + \sigma_{i-1}}(\iota (g_{i-1} a^{-1}_{i},
                                              f_i(x_{i-1})))\right) 
      & \leq & \epsilon; 
      \\
      d_\FS\left(\Phi_{T+\sigma_i} (\iota(g_i,x_i)),
           \Phi_{T+\widetilde{\tau}_{i}+\sigma_i}(\iota (g_i b_i^{-1},
                                          \widetilde{f}_i(x_i)))\right)
      & \leq & \epsilon. 
  \end{eqnarray*}
  Since $g_{i-1}a_i^{-1} = g_i b_i^{-1}$, $f_i(x_{i-1}) = \widetilde{f}_i(x_i)$
  and $T+\tau_i + \sigma_{i-1} = T+\widetilde{\tau}_{i}+\sigma_i$
  we conclude for $i = 1,\dots,n$ from the triangle inequality
  $$
  d_\FS\left(\Phi_{T+\sigma_{i-1}}(\iota( g_{i-1},x_{i-1})),
  \Phi_{T+\sigma_{i}}(\iota(g_{i}, x_{i}))\right)
  \leq 2\epsilon.  
  $$
  Using the triangle inequality we obtain
  \begin{eqnarray*}
    \lefteqn{ d_\FS( \Phi_T (  \iota(g,x) ),
              \Phi_{T+\sigma_n}(  \iota(h,y) ))   } 
    \\
    & = &
    d_\FS( \Phi_T ( \iota(g_0,x_0) ), \Phi_{T+\sigma_n} ( \iota (g_n,x_n )  ) )
    \\
    & \leq &
    \sum_{i=1}^{n} d_\FS( \Phi_{T+\sigma_{i-1}}    ( \iota(g_{i-1},x_{i-1}) ), 
                                 \Phi_{T+\sigma_{i}} ( \iota(g_{i},x_{i}) ) )
    \\
    & \leq &
    \sum_{i=1}^{n} 2\epsilon 
    \\
    & \le & 2n \epsilon.
  \end{eqnarray*}
\end{proof}

\begin{proof}
  [Proof of Proposition~\ref{prop:wide-covers-with-flow-spaces}]
  Let $S$ be a finite subset of $G$.
  Let $\widehat{N}$ be the number from Theorem~\ref{thm:long-thin-cover}.
  Let $\widetilde{N}$ and $\beta$ be the numbers (depending on $S$) 
  appearing in Definition~\ref{def:contracting-transfers}. 
  Put $\alpha := 2 \beta |S|$. 
  By Theorem~\ref{thm:long-thin-cover} there is an open $\calf$-cover 
  $\calu$ of $\FS$ of dimension at most $\widehat{N}$ and  
  $\epsilon_0 > 0$ with the property that for every 
  $z \in \FS$ there is $U_z \in \calu$ such that
  $$B_{\epsilon_0}\bigl(\Phi_{[-\alpha,\alpha]}(z)\bigr) \subseteq U_z.$$
  Put $\epsilon := \frac{\epsilon_0}{2|S|}$.
  Pick $\delta > 0$ such that~\eqref{eq:flow-uniform-continuous} holds.
  By assumption (see Definition~\ref{def:contracting-transfers}), 
  there are $T>0$, 
  a contractible compact controlled $\widetilde{N}$-dominated space $X$,
  a $G$-equivariant map $\iota \colon G \x X \to \FS$ 
  and a homotopy $S$-action $(\varphi,H)$ on $X$
  such that \eqref{nl:foliated-control} holds.
  Using Lemma~\ref{lem:foliated-controll-for-F_h} we conclude
  that for every $(g,x) \in G \x {X}$ we have
  $$\Phi_T(\iota(h,y))  \in 
    B_{\epsilon_0}\bigl(\Phi_{[-\alpha,\alpha]}(\Phi_T(\iota(g,x)))\bigr)$$
  for all $(h,y) \in S^n_{\varphi,H}(g,x)$, $n \leq |S|$.
  Hence we get  $\Phi_T(\iota(h,y)) \in U_{\Phi_T(\iota(g,x))}$ for all
  $(h,y) \in S^n_{\varphi,H}(g,x)$, $n \leq |S|$.
  This implies that 
  $\calv := \{ (\Phi_T \circ \iota)^{-1}(U) \mid U \in \calu \}$
  is $S$-long with respect to $(\varphi,H)$.
  Finally, $\dim \calv \leq \dim \calu \leq \widehat{N}$.
  Thus $G$ is transfer reducible over $\calf$,
  where we use $N:= \max\{\widehat{N},\widetilde{N}\}$.
\end{proof}


\typeout{--------- Proof of Main Theorem ------------}

\section{Non-positively curved groups are transfer reducible}
\label{sec:Non-positively_curved_groups_are_transfer_reducible}

In this section we prove our Main Theorem as stated in the introduction.
Let $G$ be a group with an isometric cocompact proper action 
on a finite dimensional $\CAT(0)$-space $X$.
We need to show that $G$ is transfer reducible over the
family $\VCyc$ of virtually cyclic subgroups,
see Definition~\ref{def:transfer-reducible}.
To this end we will show that 
Proposition~\ref{prop:wide-covers-with-flow-spaces}
applies, where $\calf = \VCyc$.

\subsection{$\overline{B}_{R}(x)$ is $(2\cdot \dim(X) + 1)$-dominated}

Fix a base point $x_0 \in X$. 
For $r \ge  0$ let 
\begin{eqnarray*}
  \rho_{r,x_0} \colon  \overline{X} 
  & \to & 
  \overline{B}_r(x_0) 
\end{eqnarray*}
be the natural projection introduced in Remark~\ref{rem:cone-topology}. 
The map $\rho_{r,x_0}$ is the identity on $\overline{B}_r(x_0)$.
If $x \in \overline{X}$ with $x \notin B_r(x_0)$, 
then $\rho_{r,x_0}(x) = c_{x_0,x}(r)$,
where $c_{x_0,x} \colon \IR \to X$ is the generalized geodesic uniquely 
determined by $c_- = 0$, $c(-\infty) = x_0$ and
$c(\infty) = x$ (see~Subsection~\ref{subsec:homotopy-action-on-B}).

\begin{lemma} 
\label{lem:X_Euclidean_neighbordood_retract}
  The space $X$ is a Euclidean neighborhood retract, i.e., there is a natural
  number $N$, a closed subset $A \subseteq \IR^N$, an open neighborhood 
  $U$ of $A$ in $\IR^N$ and a map $r \colon U \to A$ such that $r|_A = \id_A$ and
  $X$ is homeomorphic to $A$. 
  The number $N$ can be chosen to be 
  $2 \cdot \dim(X) + 1$.
\end{lemma}

\begin{proof}
  Since $X$ is proper as metric space, it is locally compact.
  Since any two points in $X$ can be joint by a unique geodesic,
  $X$ is connected and locally contractible. 
  Hence $X$ has a countable basis for its topology 
  (see~\cite[Exercise~2 in Chapter~6.5 on page~261]{Munkres(1975)}). 
  Obviously $X$ is Hausdorff.
  By assumption $\dim(X) < \infty$. We conclude 
  from~\cite[Exercise~10 in Chapter~7.9 on page~315]{Munkres(1975)}
  that $X$ is homeomorphic to a closed subset $A$ of $\IR^N$
  for $N = 2\cdot \dim(X) + 1$. 
  Now apply~\cite[Proposition~8.12 in Chapter~IV.8 on p.~83]{Dold(1980)}.
\end{proof}

\begin{lemma}
  \label{lem:balls-are-N-dominated}
  The space $\overline{B}_R(x_0)$ is a compact contractible metric
  space which is controlled $(2\cdot \dim(X) + 1)$-dominated (in the sense of
  Definition~\ref{def:N-dominated_space}).
\end{lemma}

\begin{proof} 
  Since $X$ is proper as metric space by assumption,
  the closed ball $\overline{B}_R(x_0)$  is compact.  
  The space $\overline{B}_R(x_0)$ inherits from $X$ a metric
  and is contractible.

  Because of Lemma~\ref{lem:X_Euclidean_neighbordood_retract} 
  we can find  an open subset
  $U \subset \IR^{2\cdot \dim(X) + 1}$ and maps $i \colon X \to U$ and 
  $r \colon U \to X$ with $r \circ i = \id_X$.
  Since $U$ is a smooth manifolds, it can be triangulated and hence is a
  simplicial complex of dimension $(2\cdot \dim(X) + 1)$. 
  Since $\overline{B}_R(x_0)$ is compact,
  $i\bigl(\overline{B}_R(x_0)\bigr)$ is compact and 
  hence contained in a finite subcomplex $K \subseteq U$. 
  Since $K$ and hence $r(K)$ are compact,
  we can find $C > 0$ with $r(K) \subseteq \overline{B}_{R+C}(x_0)$.
  Let 
  $$i' \colon \overline{B}_R(x_0) \to K$$ 
  be the map defined by $i$. 
  Let
  $$r' \colon K \to  \overline{B}_R(x_0)$$
  be the composite
  $$K \xrightarrow{r|_K} \overline{B}_{R+C}(x_0) 
     \xrightarrow{\rho_{R,x_0}|_{\overline{B}_{R+C}(x_0)}}
     \overline{B}_{R}(x_0).$$
  Then $r' \circ i' = \id_{\overline{B}_R(x_0)}$ and $K$ is a finite
  $(2\cdot \dim(X) + 1)$-dimensional
  simplicial complex.  
  This implies that $\overline{B}_R(x_0)$ is controlled
  $(2\cdot \dim(X) + 1)$-dominated.
\end{proof}

\subsection{Convention~\ref{conv:FS} applies to $\FS(X)$} 

Let $\FS(X)$ be the flow space for $X$ from 
Definition~\ref{def_flow_space_FS(X)}. 
We will show  that this flow space satisfies
the conditions from Convention~\ref{conv:FS}.
By Proposition~\ref{prop:cocompact} the action of $G$ on $\FS(X)$ is 
isometric, proper and cocompact.
In particular, the flow space $\FS(X)$ is locally compact.
By Proposition~\ref{prop:flow-space-is-locally-connected}
the flow space $\FS(X)$ is locally connected.
By~\cite[II.2.8(2)~on~p.179]{Bridson-Haefliger(1999)}
there is $k_G < \infty$ such that finite subgroups of $G$
have order at most $k_G$. 
By Proposition~\ref{pro:dim-for_FS-FSR}
$\dim \FS(X) - \FS(X)^\IR$ is finite.
The uniform continuity of the flow follows from
Lemma~\ref{lem:Phi_well-defined}.

\subsection{Long $\VCyc$-covers for $\FS(X)$ at periodic flow lines}
\label{subsec_long_covers_periodic_flow_lines}

We need to show that the flow space $\FS(X)$ admits
long $\VCyc$-covers at infinity and periodic flow lines
in the sense of Definition~\ref{def:at-infinity_plus_periods}.
We take for $M$ the number appearing in 
Theorem~\ref{the:covering_of_FS(X)_gamma}.
Let $\gamma > 0$ be given.
Because the action of $G$ on $\FS(X)$ is cocompact we conclude from
Theorem~\ref{the:covering_of_FS(X)_gamma} that there is $\e > 0$ and 
a $G$-invariant cofinite collection $\calv$ of 
open $\VCyc$-subsets of $\FS(X)$ such that
\begin{itemize}
  \item $\dim \calv \leq M$;
  \item for any $c \in \FS(X)_{\leq \gamma}$ there is 
     $V \in \calv$ such that $B_{2\e}(\Phi_{[-\gamma,\gamma]}(c)) \subseteq V$. 
\end{itemize}
Put $S := \bigl\{ c \in \FS \mid \exists U \in \calv \; 
               \text{such that}\; \overline{B}_\e(\Phi_{[-\gamma,\gamma]}(c))  
                                              \subseteq U\bigr\}$.
Note that $S$ is $G$-invariant, because $\calv$ is.
Moreover, $\FS(X)_{\leq \gamma} \subseteq S$.
It remains to show, that there is a compact subset $K \subseteq \FS(X)$
such that $\FS(X) - S = G \cdot K$. 
Because the action of $G$ on $\FS(X)$ is cocompact 
(Proposition~\ref{prop:cocompact}) and $\FS(X)$
is locally compact (Proposition~\ref{prop:FS-is-proper}) 
it suffices to show that $S$ is open.

Choose $U \in \calv$ with
$\overline{B}_\e(\Phi_{[-\gamma,\gamma]} (c_0)) \subseteq U$.
It suffices to show that there is $\delta > 0$ such that
$\overline{B}_\e(\Phi_{[-\gamma,\gamma]} (c)) \subseteq U$
for all $c \in \FS(X)$ with $d_\FS(c,c_0) < \delta$.
We proceed by contradiction and assume that there is no 
such $\delta$.
Then there are $c_n$, $d_n \in \FS(X)$, $t_n \in [-\gamma,\gamma]$
such that $d_\FS(c_n,c_0) < 1/n$, $d_\FS(d_n,\Phi_{t_n}(c_n)) \leq \e$
but $d_n \not\in U$. 
In particular $c_n \to c_0$. 
By passing to a subsequence we can assume that 
$t_n \to t \in [-\gamma,\gamma]$.
Thus $\Phi_{t_n}(c_n) \to \Phi_t(c_0)$.
It follows that
$d_\FS(d_n,\Phi_t(c_0))$ is bounded independent of $n$.
Because $\FS(X)$ is a proper metric space
(Proposition~\ref{prop:FS-is-proper}) 
we can pass to a further subsequence and
assume that $d_n \to d$.
Then $d_\FS(d,\Phi_t(c_0)) \leq \e$.
Thus $d \in U$.
But this implies $d_n \in U$ for sufficiently large $n$,
a contradiction.

This shows that the flow space $\FS(X)$ admits
long $\VCyc$-covers at infinity and periodic flow lines.

\subsection{Contracting transfers for $\FS(X)$}
\label{subsec:contracting_transfers}

Finally we need to show that $\FS(X)$ admits contracting transfers
in sense of Definition~\ref{def:contracting-transfers}.
Let $S$ be a finite subset of $G$ (containing $e$).
Let $\beta$ be the number appearing in Proposition~\ref{prop:flow-estimate-H}.
Set $\tilde N := 2 \dim X + 1$.
Let $\delta > 0$ be given.
Let $T, R > 0$ be the numbers coming from 
Proposition~\ref{prop:flow-estimate-H}.
Then $\overline{B}_R(x_0)$ is compact, contractible and 
controlled $\tilde N$-dominated by
Lemma~\ref{lem:balls-are-N-dominated}.
We use the homotopy $S$ action $(\varphi,H)$ on $B_R(x_0)$ from
Definition~\ref{def:homotopy-action-on-B} and the map 
$\iota \colon G \x B_r(x_0) \to  \FS(X)$ obtained by restriction
from the map from Definition~\ref{def:the_mao_iota}.
It follows from Proposition~\ref{prop:flow-estimate-H}
that~\eqref{nl:foliated-control} holds.
 
Thus Proposition~\ref{prop:wide-covers-with-flow-spaces} applies
and we conclude that $G$ is transfer reducible over $\VCyc$ as claimed
in our main Theorem~\ref{the:main_theorem} in the introduction.

\begin{remark}
  Our argument proves a little more than is stated in
  our main Theorem.
  Namely, the cover we construct is in addition cofinite
  for the action of $G$, i.e., $G \backslash \calu$ is
  finite.
  This follows from Theorem~\ref{the:covering_of_FS(X)_gamma}~\ref{the:covering_of_FS(X)_gamma:finite}
  and Theorem~\ref{thm:cover-compact-in-FS}~\ref{thm:cover-compact-in-FS:cofinite}. 
\end{remark}
   


\end{document}